\documentclass[pdflatex,sn-mathphys-num]{sn-jnl}

\usepackage{graphicx}%
\usepackage{multirow}%
\usepackage{amsmath,amssymb,amsfonts}%
\usepackage{amsthm}%
\usepackage{mathrsfs}%
\usepackage[title]{appendix}%
\usepackage{xcolor}%
\usepackage{textcomp}%
\usepackage{manyfoot}%
\usepackage{booktabs}%
\usepackage{algorithm}%
\usepackage{algorithmicx}%
\usepackage{algpseudocode}%
\usepackage{listings}%
\usepackage{latexsym, mathtools}
\usepackage{bm}

\theoremstyle{thmstyleone}%
\newtheorem{theorem}{Theorem}[section]
\newtheorem{proposition}[theorem]{Proposition}%

\theoremstyle{thmstyletwo}%
\newtheorem{problem}{Problem}[section]

\newtheorem{lemma}{Lemma}[section]

\theoremstyle{thmstylethree}%

\allowdisplaybreaks

\newcommand{\argmin}{\operatornamewithlimits{argmin}}

\begin{document}

\title[Article Title]{Mini-Batch Stochastic Halpern Algorithm for Nonexpansive Fixed Point Problems}

\author[1]{\fnm{Hideaki} \sur{Iiduka}}\email{iiduka@cs.meiji.ac.jp}

\affil*[1]{\orgdiv{Department of Computer Science}, \orgname{Meiji University}, \orgaddress{\street{1-1-1 Higashimita, Tama-ku, Kawasaki-shi}, \city{Kanagawa}, \postcode{2148571}, \country{Japan}}}

\abstract{
The Halpern algorithm is a powerful tool for finding the point in the fixed point set of a nonexpansive mapping that is closest to some given point. However, the expense of computing a nonexpansive mapping means that the algorithm may not be practical for all large-scale fixed point problems. To address this issue, a mini-batch stochastic Halpern algorithm was developed and a convergence analysis was performed to demonstrate that the algorithm with a decreasing step size and increasing batch size achieves mean-square convergence. Additionally, a convergence rate analysis was conducted to demonstrate that the convergence speed depends on the sequence of step sizes.
Furthermore, numerical experiments demonstrated the effectiveness of the proposed mini-batch stochastic Halpern algorithm.}


\maketitle

\section{Introduction}\label{sec:1}
\subsection{Fixed point approximation methods}\label{sec:1.1}
The nonexpansive fixed point problem \cite[Chapter 4]{goebel1}, \cite[Chapter 1]{goebel2}, \cite[Chapter 4]{b-c}, \cite[Chapter 3]{takahashi} to find the fixed points of a nonexpansive mapping is central among fixed point problems, comprising the convex feasibility problem \cite{bau}, monotone variational inequality problems \cite[Chapter III]{kind}, and convex minimization problems \cite[Chapter II]{hir}.

A well-known fixed point approximation method for finding a fixed point of a nonexpansive mapping $T$, that is $\bm{x}$ such that $\bm{x} = T (\bm{x})$, is the Krasnosel'ski\u\i-Mann algorithm \cite{kra,mann}. If we denote the $k$-th approximate fixed point as $\bm{x}_k \in \mathbb{R}^d$, then the Krasnosel'ski\u\i-Mann algorithm updates its iterates in direction $\bm{d}_k^{\mathrm{KM}} = T(\bm{x}_k) - \bm{x}_k$ using step size $\alpha_k \in (0,1)$ as follows:
\begin{align}\label{K_M}
\bm{x}_{k+1} 
= \bm{x}_k + \alpha_k \bm{d}_k^{\mathrm{KM}} 
= \bm{x}_k + \alpha_k (T(\bm{x}_k) - \bm{x}_k)
= (1 - \alpha_k) \bm{x}_k + \alpha_k T (\bm{x}_k).
\end{align}
Convergence of \eqref{K_M} to a fixed point of $T$ is guaranteed when $\alpha_k$ satisfies $\sum_{k=0}^{+ \infty} \alpha_k (1 - \alpha_k) = + \infty$ \cite[Theorem 5.15]{b-c}, \cite[Corollaries 1-3]{GROETSCH1972369}. It has also been shown that the algorithm achieves $\| \bm{x}_K - T (\bm{x}_K) \| = O(1/\sqrt{\sum_{k=0}^{K-1} \alpha_k (1 - \alpha_k)})$ \cite[Theorem 1]{Cominetti:2014aa}.

In contrast to the Krasnosel'ski\u\i-Mann algorithm \eqref{K_M}, which can find some fixed point of a nonexpansive mapping $T$, the Halpern algorithm \cite{halpern} finds specifically the global solution to the following convex minimization over fixed point set $\mathrm{Fix}(T) \coloneqq \{ \bm{x} \in \mathbb{R}^d \colon T(\bm{x}) = \bm{x} \}$:
\begin{align}\label{convex_0}
\text{Minimize } f_0 (\bm{x}) \coloneqq \frac{1}{2} \|\bm{x} - \bm{x}_0\|^2 
\text{ subject to } \bm{x} \in \mathrm{Fix}(T).
\end{align}
The Halpern algorithm uses direction $\bm{d}_k^{\mathrm{H}} = - \nabla f_0 (T(\bm{x}_k)) = \bm{x}_0 - T(\bm{x}_k)$ to update its iterates using step size $\alpha_k \in (0,1)$ as follows:
\begin{align}\label{H}
\bm{x}_{k+1} 
= T(\bm{x}_k) + \alpha_k \bm{d}_k^{\mathrm{H}} 
= T(\bm{x}_k) + \alpha_k (\bm{x}_0 - T(\bm{x}_k))
= \alpha_k \bm{x}_0 + (1 - \alpha_k) T (\bm{x}_k).
\end{align}
In the case of Halpern algorithm \eqref{H}, convergence to a global minimizer of $f_0$ over $\mathrm{Fix}(T)$, denoted $\bm{x}^\star$, is guaranteed when $\alpha_k$ satisfies $\lim_{k \to + \infty} \alpha_k = 0$, $\sum_{k=0}^{+ \infty} \alpha_k = + \infty$, and $\sum_{k=0}^{+ \infty} |\alpha_{k+1} - \alpha_k| < + \infty$ \cite[Theorem 2]{wit}. Solution $\bm{x}^\star$ of \eqref{convex_0} is the point in fixed point set $\mathrm{Fix}(T)$ that is closest to the point $\bm{x}_0$, called the initial point. In terms of the metric projection $P_{\mathrm{Fix}(T)}$ onto $\mathrm{Fix}(T)$, this means that $\bm{x}^\star = P_{\mathrm{Fix}(T)} (\bm{x}_0)$.

\subsection{Mini-batch stochastic fixed point approximation methods}
\label{sec:1.2}
\subsubsection{Computational difficulty for nonexpansive mapping}\label{sec:1.2.1}
As described in the previous section, the Krasnosel'ski\u\i-Mann and Halpern algorithms can be applied to nonexpansive mappings to find fixed points. However, these algorithms are not necessarily practical for all nonexpansive fixed point problems, due to their computations strongly depending on the computability of $T$.

As an example, we consider finding a fixed point for the mean of $n$ nonexpansive mappings $T_i \colon \mathbb{R}^d \to \mathbb{R}^d$ $(i=1,2,\cdots,n)$ on a $d$-dimensional Euclidean space,
\begin{align}\label{Nonexp_T}
T = \mathbb{E}_{\xi \sim \mathrm{DU}(n)}[T_\xi] = \frac{1}{n} \sum_{i=1}^n T_i,
\end{align}
where $\xi \sim \mathrm{DU}(n)$ means a random variable $\xi$ following the discrete uniform distribution on $\{1,2,\cdots,n\}$ and $\mathbb{E}_\xi [X]$ is the usual expectation of $X$ with respect to $\xi$. Then the computation of $T$ will be expensive when $n$ and $d$ are large, and existing methods \eqref{K_M} and \eqref{H} may not function properly, depending on the computational environment.

\subsubsection{Mini-batch stochastic mapping}\label{sec:1.2.2}
To address the issue of the computational difficulty for nonexpansive mapping $T = (1/n) \sum_{i=1}^n T_i$, let $b_k$ be the (mini-)batch size such that $b_k$ nonexpansive mappings are randomly chosen as samples from $\{T_1, T_2, \cdots, T_n \}$ in iteration $k$. Further, let $\xi_{k,i}$ be the random variable generated by the $i$-th sampling in iteration $k$ and $\bm{\xi}_k = (\xi_{k,1}, \xi_{k,2}, \cdots, \xi_{k,b_k})^\top$ comprise $b_k$ independent and identically distributed (i.i.d.) variables. Then the nonexpansive mapping for the $i$-th sampling in iteration $k$ can be written as $T_{\xi_{k,i}}$ and the {\em mini-batch stochastic mapping} that is the mean of $T_{\xi_{k,1}}, T_{\xi_{k,2}}, \cdots, T_{\xi_{k,b_k}}$ can be defined as follows for all $\bm{x} \in \mathbb{R}^d$:
\begin{align}\label{stochastic_mapping}
T_{\bm{\xi}_k}(\bm{x}) \coloneqq 
\frac{1}{b_k} \sum_{i=1}^{b_k} T_{\xi_{k,i}} (\bm{x}).
\end{align} 
If batch size $b_k$ is properly chosen, then $T_{\bm{\xi}_k}(\bm{x})$ is computationally feasible, despite still being dependent  on the computation environment. If $T_{\xi_{k,i}}$ is an unbiased estimator of $T$, that is $\mathbb{E}_{\xi_{k,i}}[T_{\xi_{k,i}}(\bm{x})] = T(\bm{x})$ for all $\bm{x}$, that are independent of $\xi_{k,i}$, then 
\begin{align*}
\mathbb{E}_{\bm{\xi}_k}[T_{\bm{\xi}_k}(\bm{x})]
= \frac{1}{b_k} \sum_{i=1}^{b_k} \mathbb{E}_{\xi_{k,i}} [T_{\xi_{k,i}} (\bm{x})]
= T(\bm{x}).
\end{align*}
From this, methods using \eqref{stochastic_mapping} can be expected to converge to fixed points of $T$. 

\subsubsection{Mini-batch stochastic Krasnosel'ski\u\i-Mann algorithm}\label{sec:1.2.3}
A mini-batch stochastic Krasnosel'ski\u\i-Mann algorithm \cite{minibatchKM} can be obtained by replacing term $T(\bm{x}_k)$ in \eqref{K_M} by term $T_{\bm{\xi}_k}(\bm{x}_k)$ defined in \eqref{stochastic_mapping} as follows:
\begin{align}\label{K_M_1}
\bm{x}_{k+1} 
= \bm{x}_k + \alpha_k (T_{\bm{\xi}_k} (\bm{x}_k) - \bm{x}_k)
= (1 - \alpha_k) \bm{x}_k + \alpha_k T_{\bm{\xi}_k} (\bm{x}_k).
\end{align}
It was previously proved in \cite{minibatchKM} that mini-batch stochastic Krasnosel'ski\u\i-Mann algorithm \eqref{K_M_1} converges almost surely to a fixed point of $T$ for $\alpha_k$ and $b_k$ satisfying  
\begin{align}\label{conditions_1_1}
\sum_{k=0}^{+ \infty} \alpha_k (1 - \alpha_k) = + \infty 
\text{ and } 
\sum_{k=0}^{+ \infty} \frac{1}{\sqrt{b_k}} < + \infty. 
\end{align}
It was further proved that $\min_{k \in [0:K-1]} \mathbb{E}[\| \bm{x}_k - T (\bm{x}_k) \|] = O(1/\sqrt{\sum_{k=0}^{K-1} \alpha_k (1 - \alpha_k)})$ \cite[Theorem 2]{minibatchKM}. The condition on $b_k$ in \eqref{conditions_1_1} implies that an increasing batch size is needed to guarantee almost-sure convergence of algorithm \eqref{K_M_1}.  In the case of sampling with replacement, $T_{\bm{\xi}_k} \neq T$ holds in general even when $b_k > n$. Hence, under sampling with replacement, we can use $b_k \to + \infty$ $(k \to + \infty)$ to examine the convergence of mini-batch algorithms.

For strictly increasing batch size $b_k$, there exists some $k_0 \in \mathbb{N}$ such that, for all $k \geq k_0$, $b_k \geq n$. Then it would be difficult in practice to compute $T_{\bm{\xi}_k} (\bm{x}_k) = (1/b_k) \sum_{i=1}^{b_k} T_{\xi_{k,i}} (\bm{x}_k)$ for $k \geq k_0$. However, \cite{umeda2025increasing,oowada2025faster} showed numerically that increasing the batch size a finite number of times (for example, to train ResNet-34 on ImageNet with $n = 1,281,167$, batch sizes $b_0 = 2^5, b_1 = 2^6, b_2 = 2^7$, $b_3 = 2^8$, and $b_4 = 2^9$ were used; see Section A.5 in \cite{umeda2025increasing}) accelerates mini-batch stochastic gradient descent, which is an example application of algorithm \eqref{K_M_1} (see also Section \ref{sec:1.3}). In practice,  finitely increasing batch size within some computable range will give faster convergence of algorithm \eqref{K_M_1}.

\subsection{Motivation}\label{sec:1.3}
\subsubsection{Practical benefits}\label{sec:1.3.1}
The first motivation of this paper is the practical benefit that methods based on the mini-batch stochastic mapping $T_{\bm{\xi}_k}$ defined in \eqref{stochastic_mapping} can be applied to the nonexpansive fixed point problem for the mapping $T$ defined in \eqref{Nonexp_T}, including large-scale real-world problems.

To make this motivation explicit, we investigate empirical risk minimization (ERM) problems that commonly arise in machine learning (see also discussions in \eqref{convex_t} and \eqref{convex_min}).
Let $\bm{x} \in \mathbb{R}^d$ be a parameter of a deep neural network, 
let $S = \{ (\bm{z}_1, \bm{y}_1), \ldots, (\bm{z}_n, \bm{y}_n)\}$ be the training set, where data point $\bm{z}_i$ is associated with label $\bm{y}_i$, and 
let $f_i (\bm{x}) \coloneqq f (\bm{x} ; (\bm{z}_i, \bm{y}_i))$ be the differentiable loss function corresponding to the $i$-th labeled training data $(\bm{z}_i, \bm{y}_i)$.
ERM problem is to minimize the
empirical loss defined for all $\bm{x} \in \mathbb{R}^d$ by 
$f (\bm{x}) \coloneqq (1/n) \sum_{i=1}^n f_i (\bm{x})$.
The ERM problem can be expressed as a fixed point problem for $T \coloneqq \mathrm{Id} - \eta \nabla f$ (see also \eqref{convex_t} and \eqref{convex_min}), where $\mathrm{Id}$ is the identity mapping on $\mathbb{R}^d$ and $\eta > 0$.
The simplest optimization method for the ERM problem is gradient descent (GD) defined for all $k \in \mathbb{N}$ by $\bm{x}_{k+1} = \bm{x}_k - \eta \nabla f (\bm{x}_k)$.
However, GD cannot be applied to the ERM problem, since the computation of $\nabla f$ is very expensive (see also Section \ref{sec:1.2.1}).
Meanwhile, mini-batch stochastic gradient descent (SGD) defined for all $k \in \mathbb{N}$ by 
\begin{align}\label{SGD_0}
\begin{split}
\bm{x}_{k+1}
&= T_{\bm{\xi}_k} (\bm{x}_k)
= \frac{1}{b_k} \sum_{i=1}^{b_k} T_{\xi_{k,i}} (\bm{x}_k) 
\coloneqq
\frac{1}{b_k} \sum_{i=1}^{b_k} 
\underbrace{\left( \mathrm{Id} - \eta \nabla f_{\xi_{k,i}} \right)}_{T_{\xi_{k,i}}} (\bm{x}_k)\\
&=
\bm{x}_k - \frac{\eta}{b_k} \sum_{i=1}^{b_k} \nabla f_{\xi_{k,i}} (\bm{x}_k) 
\eqqcolon
\bm{x}_k - \eta \nabla f_{\bm{\xi}_k} (\bm{x}_k)
\end{split}
\end{align}
can be applied to the ERM problem when batch size $b_k$ is properly chosen (see also Section \ref{sec:1.2.2}).
This implies that the mini-batch stochastic Krasnosel'ski\u\i-Mann algorithm \eqref{K_M_1} (i.e., $\bm{x}_{k+1} = (1-\alpha_k) \bm{x}_k + \alpha_k T_{\bm{\xi}_k} (\bm{x}_k)$) can be applied to the ERM problem.

\subsubsection{Mini-batch stochastic Halpern algorithm}\label{sec:1.3.2}
Section \ref{sec:1.2.3} indicates that the previous paper \cite{minibatchKM} demonstrated convergence of the mini-batch stochastic Krasnosel'ski\u\i-Mann algorithm \eqref{K_M_1}.
Meanwhile, it would be insufficient to investigate the mini-batch stochastic Halpern algorithm that replaces the $T(\bm{x}_k)$ term in \eqref{H} by $T_{\bm{\xi}_k}(\bm{x}_k)$, defined in \eqref{stochastic_mapping}:
\begin{align}\label{H_1}
\bm{x}_{k+1} 
= T_{\bm{\xi}_k} (\bm{x}_k) + \alpha_k (\bm{x}_0 - T_{\bm{\xi}_k} (\bm{x}_k))
= \alpha_k \bm{x}_0 + (1 - \alpha_k) T_{\bm{\xi}_k} (\bm{x}_k).
\end{align}
As indicated in Section \ref{sec:1.1}, the Halpern algorithm \eqref{H} can find the point $\bm{x}^\star$ in the fixed point set $\mathrm{Fix}(T)$ that is closest to the initial point $\bm{x}_0$.
However, the Halpern algorithm \eqref{H} using the mapping $T$ cannot be applied to practical problems, such as ERM problems (see Section \ref{sec:1.3.1}).
Moreover, although the mini-batch stochastic Krasnosel'ski\u\i-Mann algorithm \eqref{K_M_1} can be applied to ERM problems, it is not guaranteed that the algorithm \eqref{K_M_1} can find the point $\bm{x}^\star$.
The second motivation behind this work is thus to show that
the mini-batch stochastic Halpern algorithm \eqref{H_1} can solve the convex minimization problem over fixed point set \eqref{convex_0}.  

\subsection{Contribution}\label{sec:1.4}
\subsubsection{Convergence of mini-batch stochastic Halpern algorithm}
The main contribution of the paper is to demonstrate convergence analysis (Theorem \ref{thm:1}) and convergence rate analysis (Theorem \ref{thm:2}) of the mini-batch stochastic Halpern algorithm \eqref{H_1}.
Herein, three lemmas (Lemmas \ref{lem:1}, \ref{lem:2}, and \ref{lem:3}) are used to prove that if $(\alpha_k)$ and $(b_k)$ satisfy
\begin{align}\label{condition_1_0}
\begin{split}
&\lim_{k \to + \infty} \alpha_k = 0, \text{ } 
\sum_{k=0}^{+\infty} \alpha_k = + \infty, \text{ }
\sum_{k=0}^{+\infty} | \alpha_{k+1} - \alpha_k| < + \infty,\\
&\frac{1}{b_k} \leq \alpha_k^2 \text{ } (k \in \{0\} \cup \mathbb{N}), \text{ and } 
\sum_{k=0}^{+ \infty} \frac{1}{\sqrt{b_k}} < + \infty,
\end{split}
\end{align}
then algorithm \eqref{H_1} converges in mean square to $\bm{x}^\star = P_{\mathrm{Fix}(T)}(\bm{x}_0)$ (Theorem \ref{thm:1}). The $b_k$ condition in \eqref{condition_1_0} implies that an increasing batch size is needed to guarantee the mean-square convergence of algorithm \eqref{H_1}. The same condition to guarantee the almost-sure convergence of algorithm \eqref{K_M_1} was given as \eqref{conditions_1_1}.

It will also be shown that, under certain additional conditions, the convergence rate of algorithm \eqref{H_1} is given by the following (Theorem \ref{thm:2}): 
\begin{align*}
\min_{k \in [0:K-1]} \mathbb{E} \left[ f_0 (\bm{x}_k) \right]
= 
f_0^\star 
+
\begin{dcases}
O \left( \frac{\sum_{k=0}^{K-1} \alpha_k^2}{\sum_{k=0}^{K-1} \alpha_k} \right) 
&\text{ } \left(  \sum_{k=0}^{K-1} \alpha_k^2 = o \left( \sum_{k=0}^{K-1} \alpha_k  \right) \right)\\
O \left( \frac{1}{\sum_{k=0}^{K-1} \alpha_k} \right)
&\text{ } \left( \sum_{k=0}^{+ \infty} \alpha_k^2 < + \infty \right), 
\end{dcases}
\end{align*}
where $f_0^\star = f_0 (\bm{x}^\star)$ is the optimal value of $f_0$ over $\mathrm{Fix}(T)$. For example, with decreasing step size $\alpha_k = 1/\sqrt{k+1}$ and some increasing sequence of batch sizes $b_k$, this becomes 
\begin{align*}
\min_{k \in [0:K-1]} \mathbb{E} \left[ f_0 (\bm{x}_k) \right]
= 
f_0^\star 
+
O \left( \frac{\log K}{\sqrt{K}} \right).
\end{align*} 
This $O ( \log K/\sqrt{K})$ convergence of algorithm \eqref{H_1} for this example is competitive with those of existing methods (see, e.g., \cite[Chapter 2]{hazan2022}).

\subsubsection{Numerical results}
To demonstrate the effectiveness of the proposed mini-batch stochastic Halpern algorithm, we conduct experiments on the CIFAR-100 dataset \cite{krizhevsky2009learning} using a residual neural network \cite{he2016} (Section \ref{sec:4})
and compare the mini-batch stochastic Krasnosel'ski\u\i--Mann algorithm \eqref{K_M_1} with the mini-batch stochastic Halpern algorithm \eqref{H_1} in terms of their performance on the ERM problem (see also Section \ref{sec:1.3.1}).
Our numerical results demonstrate that the mini-batch stochastic Halpern algorithm \eqref{H_1} with $(\alpha_k)$ and $(b_k)$ satisfying \eqref{condition_1_0} achieves higher test accuracy than the mini-batch stochastic Krasnosel'ski\u\i--Mann algorithm \eqref{K_M_1}.

\subsection*{Notation and definitions}
Let $\mathbb{N}$ be the set of natural numbers and define $[n] \coloneqq \{1,2,\cdots, n\}$ and $[0:n] \coloneqq \{0,1,\cdots, n\}$ for $n \in \mathbb{N}$. Further, $\mathbb{R}^d$ denotes the usually $d$-dimensional Euclidean space with inner product $\langle \cdot, \cdot \rangle$ and norm $\|\cdot\|$ and we define $\mathbb{R}_+^d \coloneqq \{ \bm{x} = (x_i)_{i=1}^d \in \mathbb{R}^d \colon x_i \geq 0 \text{ } (i\in [n]) \}$. Let $\mathrm{Id}$ (i.e., $\mathrm{Id}(\bm{x}) \coloneqq \bm{x}$ for all $\bm{x} \in \mathbb{R}^d$) denote the identity mapping on $\mathbb{R}^d$, and $B_r (\bm{x}) \coloneqq \{ \bm{y} \in \mathbb{R}^d \colon \|\bm{y} - \bm{x} \| \leq r \}$ denote the closed ball with center $\bm{x} \in \mathbb{R}^d$ and radius $r > 0$.

For a mapping $T \colon \mathbb{R}^d \to \mathbb{R}^d$, we define the fixed point set of $T$ as $\mathrm{Fix} (T) \coloneqq \{ \bm{x} \in \mathbb{R}^d \colon \bm{x} = T(\bm{x}) \}$. A mapping $T$ is nonexpansive if $\| T(\bm{x}) - T(\bm{y})\| \leq \|\bm{x} - \bm{y}\|$ for all $\bm{x}, \bm{y} \in \mathbb{R}^d$. The fixed point set of such a nonexpansive mapping has been shown to be closed and convex \cite[Proposition 5.3]{goebel1}. A mapping $T$ is firmly nonexpansive if $\| T(\bm{x}) - T(\bm{y})\|^2 \leq \langle T(\bm{x}) - T (\bm{y}), \bm{x} - \bm{y}\rangle$ for all $\bm{x}, \bm{y} \in \mathbb{R}^d$. Since the Cauchy-Schwarz inequality means that a firmly nonexpansive mapping $T$ satisfies $\| T(\bm{x}) - T(\bm{y})\|^2 \leq \| T(\bm{x}) - T (\bm{y}) \| \| \bm{x} - \bm{y} \|$, firm nonexpansivity guarantees  nonexpansivity. The metric projection $P_C$ onto a nonempty, closed convex set $C$ $(\subset \mathbb{R}^d)$ is defined for all $\bm{x} \in \mathbb{R}^d$ by  $P_C (\bm{x}) \in C$ and $\| \bm{x} - P_C (\bm{x}) \| = \inf_{\bm{y} \in \mathbb{R}^d} \| \bm{x} - \bm{y} \|$. Projection $P_C$ is firmly nonexpansive with fixed point set $\mathrm{Fix}(P_C) = C$ \cite[Theorem 3.1.4(i)]{takahashi}, \cite[p.371]{bau}.

Let $\mathrm{P}(A)$ denote the probability of event $A$ and $\mathbb{E}_\xi [\bm{X}(\xi)]$ denote the expectation of a random variable $\bm{X}(\xi)$ with respect to the random variable $\xi$. Then the variance of $\bm{X}(\xi)$ with respect to $\xi$ is can be expressed as $\mathbb{V}_\xi [\bm{X}(\xi)] \coloneqq \mathbb{E}_\xi [\| \bm{X}(\xi) - \mathbb{E}_\xi [\bm{X}(\xi)] \|^2]$. Further, let $\mathbb{E}_\xi [\bm{X}(\xi)|\bm{Y}]$ and $\mathbb{V}_\xi [\bm{X}(\xi)|\bm{Y}]$ denote respectively the expectation and variance of $\bm{X}(\xi)$ conditioned on $\bm{Y}$. Given independent $\bm{\xi}_0, \bm{\xi}_1, \cdots, \bm{\xi}_k$, the total expectation is defined as $\mathbb{E} \coloneqq \mathbb{E}_{\bm{\xi}_0} \mathbb{E}_{\bm{\xi}_1} \cdots \mathbb{E}_{\bm{\xi}_k}$. Finally, let $\xi \sim \mathrm{DU}(n)$ denote that $\xi$ follows the discrete uniform distribution on $[n]$.

\section{Stochastic Nonexpansive Fixed Point Problem}
\label{sec:2}
Let us consider the following problem \cite[Problem 1.1 and Assumption 2.1]{minibatchKM}.

\begin{problem}\label{prob:1}
Let $n \in \mathbb{N}$ and $T_i \colon \mathbb{R}^d \to \mathbb{R}^d$ ($i\in [n]$). Furthermore, let $\xi$ be a random variable taking values in $[n]$. A stochastic mapping $T_\xi \colon \mathbb{R}^d \to \mathbb{R}^d$ is randomly chosen in $\{ T_i \}_{i=1}^{n}$ and a mapping $T \colon \mathbb{R}^d \to \mathbb{R}^d$ is defined for all $\bm{x} \in \mathbb{R}^d$ by 
\begin{align}\label{true_t}
T (\bm{x}) \coloneqq \mathbb{E}_\xi [T_\xi (\bm{x})],
\end{align}
where $\xi$ is independent of $\bm{x}$. We assume that:
\begin{enumerate}
\item[(A1)] [Nonexpansivity] $T_i$ $(i\in [n])$ is nonexpansive, i.e., $\|T_i (\bm{x}) - T_i (\bm{y})\| \leq \|\bm{x} - \bm{y}\|$ $(\bm{x}, \bm{y} \in \mathbb{R}^d)$;
\item[(A2)] [Boundedness of variance of stochastic mapping] There exists $\sigma \geq 0$ such that, for all $\bm{x} \in \mathbb{R}^{d}$, 
\begin{align*}
\mathbb{V}_{\xi}[T_{\xi}(\bm{x})] \coloneqq \mathbb{E}_{\xi} \big[\| T_{\xi}(\bm{x}) - \underbrace{\mathbb{E}_{\xi}[T_{\xi}(\bm{x})]}_{T (\bm{x})} \|^2 \big] \leq \sigma^{2};
\end{align*}
\item[(A3)] [Nonemptiness of fixed point set] $\mathrm{Fix}(T) \neq \emptyset$.
\end{enumerate}
Accordingly, we would like to find a fixed point $\bm{x}^\star$ of $T$, i.e., 
\begin{align*}
\bm{x}^\star = T (\bm{x}^\star) = \mathbb{E}_\xi [T_\xi (\bm{x}^\star)].
\end{align*} 
\end{problem}

\subsection{Discussion of assumptions with examples}
Assumption (A1) guarantees that $T$ defined by \eqref{true_t} is nonexpansive. From definition \eqref{true_t} and the properties  of $\mathbb{E}_\xi$, the following relations hold:
\begin{align*}
\| T(\bm{x}) - T (\bm{y})  \|
&= 
\| \mathbb{E}_\xi [T_\xi (\bm{x}) ] - \mathbb{E}_\xi [T_\xi (\bm{y}) ] \|
= 
\| \mathbb{E}_\xi [T_\xi (\bm{x}) - T_\xi (\bm{y}) ] \|.
\end{align*}
These combined with Jensen's inequality and (A1) imply 
\begin{align}\label{nonexp_T}
\| T(\bm{x}) - T (\bm{y})  \|
\leq
\mathbb{E}_\xi [\| T_\xi (\bm{x}) - T_\xi (\bm{y}) \|]
\leq 
\| \bm{x} - \bm{y} \|.
\end{align}
Let us examine two examples of $T_\xi$ and $T$ satisfying Assumption (A2) \cite[Example 2.1]{minibatchKM}. As our first example, consider the convex feasible problem of taking the intersection $C$ of multiple closed convex sets $C_i$. For this example, the following mapping $T$ is used:
\begin{align}\label{cfp_t}
T (\bm{x})
= 
\mathbb{E}_{\xi \sim \mathrm{DU}(n)} [\underbrace{P_\xi (\bm{x})}_{T_\xi (\bm{x})}]
= \frac{1}{n} \sum_{i=1}^n P_i (\bm{x}),
\end{align}
where $\bm{x} \in \mathbb{R}^d$ and $T_i = P_i$ is the metric projection onto $C_i$. Then, for $C \coloneqq \bigcap_{i=1}^n C_i = \bigcap_{i=1}^n\mathrm{Fix}(T_i) \neq \emptyset$, it follows from Proposition 4.47 in \cite{b-c} and \eqref{cfp_t} that 
\begin{align}\label{convex_feasibility_p}
\mathrm{Fix}(T) 
= \mathrm{Fix} \left( \frac{1}{n} \sum_{i=1}^n P_i  \right)
= \mathrm{Fix} \left( \frac{1}{n} \sum_{i=1}^n T_i  \right)
= \bigcap_{i=1}^n \mathrm{Fix}(T_i)
= \bigcap_{i=1}^n C_i = C.
\end{align}
Further, from the firm nonexpansivity of $P_i$ $(i\in [n])$ and the triangle inequality, we know that, for all $\bm{x}^\star \in C$ and all $\bm{x} \in B_{r} (\bm{x}^\star)$,
\begin{align}\label{norm_p_i}
\|P_i (\bm{x})\| 
\leq 
\| P_i (\bm{x}) - P_i (\bm{x}^\star)  \| + \|\bm{x}^\star\|
\leq \|\bm{x} - \bm{x}^\star\| + \|\bm{x}^\star\| \leq r + \|\bm{x}^\star\|.
\end{align}
By applying \eqref{norm_p_i}, we can obtain that 
\begin{align*}
\mathbb{V}_{\xi \sim \mathrm{DU}(n)}[P_\xi (\bm{x})]
&\leq 
\mathbb{E}_{\xi \sim \mathrm{DU}(n)} \left[\| P_\xi (\bm{x})\|^2 \right]
= 
\sum_{i=1}^n \|P_i (\bm{x})\|^2 \mathrm{P}(\xi = i)
\leq  
\left( r + \|\bm{x}^\star\| \right)^{2}.
\end{align*}

As our second example, we consider the convex minimization problem of finding a global minimizer of function $f$ defined as the mean of multiple convex functions $f_i$, $f \coloneqq (1/n) \sum_{i=1}^n f_i$. Then we let $\nabla f_\xi \colon \mathbb{R}^d \to \mathbb{R}^d$ be the stochastic gradient of $f$ and assume that the following conditions hold:
\begin{enumerate}
\item[(C1)] [Unbiasedness of stochastic gradient] For all $\bm{x} \in \mathbb{R}^d$, $\mathbb{E}_{\xi \sim \mathrm{DU}(n)} [\nabla f_\xi (\bm{x})] = \nabla f (\bm{x})$;
\item[(C2)] [Boundedness of variance of stochastic gradient] There exists $\sigma_g \geq 0$ such that, for all $\bm{x} \in \mathbb{R}^{d}$, $\mathbb{V}_{\xi \sim \mathrm{DU}(n)}[\nabla f_{\xi}(\bm{x})] \leq \sigma_g^{2}$.
\end{enumerate} 
In this case, we define our mapping $T$ for $\bm{x} \in \mathbb{R}^d$ as
\begin{align}\label{convex_t}
T (\bm{x})
= 
\mathbb{E}_{\xi \sim \mathrm{DU}(n)} [\underbrace{(\mathrm{Id} - \eta \nabla f_\xi) (\bm{x})}_{T_\xi (\bm{x})}]
= \bm{x} - \eta \nabla f(\bm{x}),
\end{align}
where $\eta > 0$. Then it follows from \eqref{convex_t} and the convexity of $f$ that 
\begin{align}\label{convex_min}
\mathrm{Fix}(T) 
= \argmin_{\bm{x} \in \mathbb{R}^d} f(\bm{x})
= \argmin_{\bm{x} \in \mathbb{R}^d} \frac{1}{n} \sum_{i=1}^n f_i (\bm{x})
\neq 
\bigcap_{i=1}^n \argmin_{\bm{x} \in \mathbb{R}^d} f_i (\bm{x})
= 
\bigcap_{i=1}^n \mathrm{Fix}(T_i).
\end{align}
It should be noted here that $\mathrm{Fix}(T) = \argmin_{\bm{x} \in \mathbb{R}^d} f(\bm{x})$ in \eqref{convex_min} holds without assuming $\bigcap_{i=1}^n \mathrm{Fix}(T_i) = \bigcap_{i=1}^n \argmin_{\bm{x} \in \mathbb{R}^d} f_i (\bm{x}) \neq \emptyset$, which differs from the convex feasibility problem considered in \eqref{convex_feasibility_p}. From (C1) and (C2), it follows that 
\begin{align*}
\mathbb{V}_{\xi \sim \mathrm{DU}(n)}[\bm{x} - \eta \nabla f_\xi (\bm{x})]
&= 
\mathbb{E}_{\xi \sim \mathrm{DU}(n)}
[ \| ( \bm{x} - \eta \nabla f_\xi (\bm{x}) ) - ( \bm{x} - \eta \nabla f (\bm{x}) ) \|^{2} ]\\
&= 
\eta^{2}
\mathbb{E}_{\xi \sim \mathrm{DU}(n)} [ \|\nabla f_\xi (\bm{x}) - \nabla f (\bm{x})\|^{2}]\\
&=
\eta^{2} \mathbb{V}_{\xi \sim \mathrm{DU}(n)}
[\nabla f_\xi (\bm{x})] 
\leq 
(\eta \sigma_g)^{2}.
\end{align*}
It has previously been shown that if $\eta$ is suitably restricted, then $T_i \coloneqq \mathrm{Id} - \eta \nabla f_i$ is nonexpansive \cite[Proposition 2.3]{iiduka_JOTA}.

\subsection{Mini-batch stochastic mapping}
We define batch size $b \in \mathbb{N}$ as the number of samples $\bm{\xi} = (\xi_{1}, \xi_{2}, \cdots, \xi_{b})^\top$, which are i.i.d. variables. Then we define the {\em mini-batch stochastic mapping} of $T$ in \eqref{true_t} using the following as $T_{\bm{\xi}}$:
\begin{align}\label{mini_batch}
T_{\bm{\xi}}(\bm{x}) \coloneqq 
\frac{1}{b} \sum_{i=1}^b T_{\xi_{i}}(\bm{x}),
\end{align} 
where $\bm{x} \in \mathbb{R}^d$ is independent of $\bm{\xi}$. As demonstrated by the following proposition, mini-batch stochastic mapping $T_{\bm{\xi}}$ inherits useful properties of stochastic mapping $T_\xi$, such as being unbiased (\eqref{true_t} in Problem \ref{prob:1}) and having a bounded variance (Assumption (A2)). The proof of proposition \ref{prop:1} was given in \cite{minibatchKM}.

\begin{proposition}\label{prop:1}
For Problem \ref{prob:1} and the mini-batch stochastic mapping $T_{\bm{\xi}}$ defined by (\ref{mini_batch}), if $\bm{x}, \bm{y} \in \mathbb{R}^d$ are independent of $\bm{\xi}$, then the following properties hold.
\begin{enumerate}
\item[{\em (i)}] {\em [Unbiasedness of mini-batch stochastic mapping]} $\displaystyle{\mathbb{E}_{\bm{\xi}}[T_{\bm{\xi}}(\bm{x})] = T(\bm{x})}$;
\item[{\em (ii)}] {\em [Boundedness of variance of mini-batch stochastic mapping]} $\displaystyle{\mathbb{V}_{\bm{\xi}}[T_{\bm{\xi}} (\bm{x})] \leq \frac{\sigma^2}{b}}$.
\end{enumerate}
Moreover, {\em (i)} and {\em (ii)} together imply that 
\begin{align}\label{app_nonexp}
\mathbb{E}_{\bm{\xi}} \left[ \| T_{\bm{\xi}}(\bm{x}) - T(\bm{y}) \|^2 \right]
\leq 
\|\bm{x} - \bm{y} \|^2 
+ 
\frac{\sigma^2}{b}.
\end{align}
\end{proposition}

Inequality \eqref{app_nonexp} of the proposition implies that if batch size $b$ is sufficiently large, then the following relations hold:
\begin{align*}
&\mathbb{E}_{\bm{\xi}} \left[ \| T_{\bm{\xi}}(\bm{x}) - T(\bm{x}) \|^2 \right]
\leq 
\|\bm{x} - \bm{x} \|^2 
+ 
\frac{\sigma^2}{b}
\approx 
0,\\
&\mathbb{E}_{\bm{\xi}} \left[ \| T_{\bm{\xi}}(\bm{x}) - T(\bm{y}) \|^2 \right]
\leq 
\|\bm{x} - \bm{y} \|^2 
+ 
\frac{\sigma^2}{b}
\approx 
\|\bm{x} - \bm{y} \|^2,
\end{align*}
which imply that $T_{\bm{\xi}}$ with a sufficiently large $b$ well  approximates the nonexpansive mapping $T$ (where the nonexpansivity of $T$ follows from \eqref{nonexp_T}).

\section{Mini-Batch Stochastic Halpern Algorithm}
\label{sec:3}
As previously (Section \ref{sec:1.2}), $\xi_{k,i}$ is assumed to be the random variable generated by the $i$-th sampling in iteration $k$. Then $\bm{x}_k$ is independent of $\bm{\xi}_k$ from being computed before the sampling of $\bm{\xi}_k = (\xi_{k,1}, \xi_{k,2}, \cdots, \xi_{k,b_k})^\top$. It follows from \eqref{mini_batch} that the mini-batch stochastic mapping of $T$ at iteration $k$ can be defined as follows:
\begin{align}\label{mini_batch_t}
T_{\bm{\xi}_k} (\bm{x}) \coloneqq \frac{1}{b_k} \sum_{i=1}^{b_k} T_{\xi_{k,i}}(\bm{x}).
\end{align}
The pseudo-code for the mini-batch stochastic Halpern algorithm is given below as Algorithm 1.

\begin{algorithm} 
\caption{Mini-Batch Stochastic Halpern Algorithm} 
\label{algo:1} 
\begin{algorithmic}[1] 
\Require
$\bm{x}_{0} \in\mathbb{R}^d$ (initial point), $\alpha_k \in (0,1]$ (step size), $b_k \in \mathbb{N}$ (batch size), $K \in \mathbb{N}$ (steps).
\Ensure
$\bm{x}_{K}$
\For{$k = 0, 1, \cdots, K-1$}
\State
$\bm{\xi}_k = (\xi_{k,1}, \xi_{k,2}, \cdots, \xi_{k,b_k})^\top$
\State
$T_{\bm{\xi}_k}(\bm{x}_k) \coloneqq \frac{1}{b_k} \sum_{i=1}^{b_k} T_{\xi_{k,i}}(\bm{x}_k)$
\State
$\bm{x}_{k+1} \coloneqq 
T_{\bm{\xi}_k} (\bm{x}_k) + \alpha_k (\bm{x}_0 - T_{\bm{\xi}_k}(\bm{x}_k))
=
\alpha_k \bm{x}_0 + (1-\alpha_k) T_{\bm{\xi}_k}(\bm{x}_k)$
\State
$k \gets k + 1$
\EndFor
\end{algorithmic}
\end{algorithm}

In theory, we could assume sampling with replacement, in which case $T_{\bm{\xi}_k} \neq T$ holds even if batch size $b_k$ is greater than $n$, so in an examination of the convergence of mini-batch algorithms under sampling with replacement, we can use $b_k \to + \infty$ $(k \to + \infty)$.

\subsection{Convergence}
\label{sec:3.1}
First, we show that the sequence $(\bm{x}_k)$ is bounded in the sense of the expected squared norm.

\begin{lemma}\label{lem:1}
The sequence $(\bm{x}_k)$ generated by Algorithm \ref{algo:1} for Problem \ref{prob:1} satisfies, for all $k \in \{0\} \cup \mathbb{N}$ and all $\bm{x}^\star \in \mathrm{Fix}(T)$,
\begin{align*}
\mathbb{E}_{\bm{\xi}_k} \left[ \| \bm{x}_{k+1} - \bm{x}^\star \|^2 \Big| \bm{\xi}_{[k-1]} \right]
\leq 
(1 - \alpha_k) \| \bm{x}_{k} - \bm{x}^\star \|^2 
+ \alpha_k \|\bm{x}_0 - \bm{x}^\star \|^2
+ \frac{\sigma^2}{b_k},
\end{align*}
where $\bm{\xi}_{[k-1]} \coloneqq \{\bm{\xi}_0, \bm{\xi}_1, \cdots, \bm{\xi}_{k-1} \}$. This implies that, if
\begin{align}\label{conditions_b_a_1}
\frac{1}{b_k} \leq \alpha_k,
\end{align}
for all $k \in \{0\} \cup \mathbb{N}$, then there exists $M \geq \|\bm{x}_0 - \bm{x}^\star\|^2 + \sigma^2$ such that, for all $k \in \{0\} \cup \mathbb{N}$, 
\begin{align*}  
\mathbb{E} \left[ \|\bm{x}_k - \bm{x}^\star \|^2 \right] \leq M 
\text{ and }
\mathbb{E} \left[ \|T_{\bm{\xi}_k} (\bm{x}_k) - \bm{x}^\star \|^2 \right] \leq M + \sigma^2.
\end{align*} 
\end{lemma}

\begin{proof}
The convexity of $\| \cdot \|^2$ and the definition of $\bm{x}_{k+1}$ (Step 4 of Algorithm \ref{algo:1}) imply that
\begin{align*}
&\mathbb{E}_{\bm{\xi}_k} \left[ \| \bm{x}_{k+1} - \bm{x}^\star \|^2 \Big| \bm{\xi}_{[k-1]} \right]\\
&= 
\mathbb{E}_{\bm{\xi}_k} \left[ \| \alpha_k (\bm{x}_{0} - \bm{x}^\star)
+ (1-\alpha_k) (T_{\bm{\xi}_k} (\bm{x}_{k}) - \bm{x}^\star) \|^2 \Big| \bm{\xi}_{[k-1]} \right]\\
&\leq
\alpha_k \mathbb{E}_{\bm{\xi}_k} \left[ \| \bm{x}_{0} - \bm{x}^\star \|^2 \Big| \bm{\xi}_{[k-1]} \right]
+ 
(1 - \alpha_k) \mathbb{E}_{\bm{\xi}_k} \left[ \| T_{\bm{\xi}_k} (\bm{x}_{k}) - \bm{x}^\star \|^2 \Big| \bm{\xi}_{[k-1]} \right]. 
\end{align*}
Since Proposition \ref{prop:1} with $\bm{y} = \bm{x}^\star = T(\bm{x}^\star)$ ensures that
\begin{align}\label{ineq_prop}
\mathbb{E}_{\bm{\xi}_k} \left[ \| T_{\bm{\xi}_k} (\bm{x}_{k}) - \bm{x}^\star \|^2 \Big| \bm{\xi}_{[k-1]} \right]
\leq 
\| \bm{x}_{k} - \bm{x}^\star \|^2 + \frac{\sigma^2}{b_k},
\end{align} 
we have
\begin{align*}
\mathbb{E}_{\bm{\xi}_k} \left[ \| \bm{x}_{k+1} - \bm{x}^\star \|^2 \Big| \bm{\xi}_{[k-1]} \right]
&\leq 
\alpha_k  \| \bm{x}_{0} - \bm{x}^\star \|^2 
+ 
(1 - \alpha_k) \left( \| \bm{x}_{k} - \bm{x}^\star \|^2 + \frac{\sigma^2}{b_k} \right)\\
&\leq 
(1 - \alpha_k) \| \bm{x}_{k} - \bm{x}^\star \|^2 + \alpha_k \| \bm{x}_{0} - \bm{x}^\star \|^2 + \frac{\sigma^2}{b_k}.
\end{align*}
The total expectation $\mathbb{E} \coloneqq \mathbb{E}_{\bm{\xi}_0} \mathbb{E}_{\bm{\xi}_1} \cdots \mathbb{E}_{\bm{\xi}_{k-1}} \mathbb{E}_{\bm{\xi}_k}$ for the above inequality, together with the condition $1/ b_k \leq \alpha_k$, implies that 
\begin{align}\label{bdd_1}
\begin{split}
\mathbb{E} \left[ \| \bm{x}_{k+1} - \bm{x}^\star \|^2 \right]
&\leq 
(1 - \alpha_k) \mathbb{E} \left[ \| \bm{x}_{k} - \bm{x}^\star \|^2 \right] + \alpha_k \| \bm{x}_{0} - \bm{x}^\star \|^2 + \frac{\sigma^2}{b_k}\\
&\leq 
(1 - \alpha_k) \mathbb{E} \left[ \| \bm{x}_{k} - \bm{x}^\star \|^2 \right] + \alpha_k \left( \| \bm{x}_{0} - \bm{x}^\star \|^2 + \sigma^2 \right).
\end{split}
\end{align}
Here, we can set $M \geq 0$ such that $M \geq \| \bm{x}_{0} - \bm{x}^\star \|^2 + \sigma^2$ and then show by mathematical induction that $\mathbb{E}[ \| \bm{x}_{k+1} - \bm{x}^\star \|^2] \leq M$ holds for all $k \in \{0\} \cup \mathbb{N}$. Since \eqref{bdd_1} with $k = 0$ leads to the finding that 
\begin{align*}
\mathbb{E} \left[ \| \bm{x}_{1} - \bm{x}^\star \|^2 \right]
&\leq 
(1 - \alpha_0) \| \bm{x}_{0} - \bm{x}^\star \|^2 + \alpha_0 \left( \| \bm{x}_{0} - \bm{x}^\star \|^2 + \sigma^2 \right)\\
&\leq 
(1 - \alpha_0) M + \alpha_0 M = M. 
\end{align*}
We assume that $\mathbb{E}[ \| \bm{x}_{m} - \bm{x}^\star \|^2] \leq M$ holds for some $m \in \mathbb{N}$. Then, from \eqref{bdd_1} with $k = m$, we have
\begin{align*}
\mathbb{E} \left[ \| \bm{x}_{m+1} - \bm{x}^\star \|^2 \right]
&\leq 
(1 - \alpha_m) \mathbb{E} \left[ \| \bm{x}_{m} - \bm{x}^\star \|^2 \right] + \alpha_m \left( \| \bm{x}_{0} - \bm{x}^\star \|^2 + \sigma^2 \right)\\
&\leq 
(1 - \alpha_m) M + \alpha_m M = M,
\end{align*}
which implies that $\mathbb{E}[ \| \bm{x}_{k+1} - \bm{x}^\star \|^2] \leq M$ for all $k \in \{0\} \cup \mathbb{N}$. From \eqref{ineq_prop}, we have, for all $k \in \{0\} \cup \mathbb{N}$, 
\begin{align*}
\mathbb{E} \left[ \| T_{\bm{\xi}_k} (\bm{x}_{k}) - \bm{x}^\star \|^2  \right]
\leq 
\mathbb{E} \left[ \| \bm{x}_{k} - \bm{x}^\star \|^2 \right] + \frac{\sigma^2}{b_k}
\leq 
M + \sigma^2 \alpha_k 
\leq 
M + \sigma^2.
\end{align*}
This completes the proof.
\end{proof}

Next, we show that the distance between $\bm{x}_k$ and $\bm{x}_{k+1}$ converges to $0$ in the sense of the expected norm.

\begin{lemma}\label{lem:2}
The sequence $(\bm{x}_k)$ generated by Algorithm \ref{algo:1} under \eqref{conditions_b_a_1} for Problem \ref{prob:1} satisfies 
\begin{align*}
\mathbb{E} [ \| \bm{x}_{k+1} - \bm{x}_{k} \| ]
\leq
(1 - \alpha_k) \mathbb{E} \left[ \|\bm{x}_k - \bm{x}_{k-1} \| \right]
+ 
M_1 |\alpha_k - \alpha_{k-1}|
+
\sigma \left( \frac{1}{\sqrt{b_k}} + \frac{1}{\sqrt{b_{k-1}}} \right),
\end{align*}
for all $k \in \{0\} \cup \mathbb{N}$, where $M_1 \coloneqq \|\bm{x}_0 \| + \sqrt{2 (M+\|\bm{x}^\star\|^2 + \sigma^2)}$, which implies that, if 
\begin{align}\label{conditions_b_a_2}
\sum_{k=0}^{+ \infty} \alpha_k = + \infty, \text{ }
\sum_{k=0}^{+ \infty} | \alpha_{k+1} - \alpha_k | < + \infty, \text{ and }
\sum_{k=0}^{+ \infty} \frac{1}{\sqrt{b_k}} < + \infty,
\end{align}
then
\begin{align*}
\lim_{k \to + \infty} \mathbb{E} [\| \bm{x}_{k+1} - \bm{x}_k\| ] = 0.
\end{align*}
\end{lemma}

\begin{proof}
The definition of $\bm{x}_{k+1}$ (Step 4 of Algorithm \ref{algo:1}) and the triangle inequality imply that 
\begin{align*}
&\mathbb{E}_{\bm{\xi}_k} \left[ \| \bm{x}_{k+1} - \bm{x}_k \| \Big| \bm{\xi}_{[k-1]} \right]\\
&= 
\mathbb{E}_{\bm{\xi}_k} \left[ \left\| 
\left( \alpha_k \bm{x}_{0} 
+ (1-\alpha_k) T_{\bm{\xi}_k} (\bm{x}_{k}) 
\right)
-
\left( \alpha_{k-1} \bm{x}_{0} 
+ (1-\alpha_{k-1}) T_{\bm{\xi}_{k-1}} (\bm{x}_{k-1}) 
\right) \right\| \Big| \bm{\xi}_{[k-1]} \right]\\
&= 
\mathbb{E}_{\bm{\xi}_k} \big[
\big\|
(\alpha_k - \alpha_{k-1}) \bm{x}_0
+ 
(1 - \alpha_k) (T_{\bm{\xi}_k} (\bm{x}_{k}) - T_{\bm{\xi}_{k-1}} (\bm{x}_{k-1}) )\\
&\quad +
(\alpha_{k-1} - \alpha_k ) T_{\bm{\xi}_{k-1}} (\bm{x}_{k-1})
\big\| \big| \bm{\xi}_{[k-1]} \big]\\
&\leq
(1 - \alpha_k) 
\underbrace{\mathbb{E}_{\bm{\xi}_k} 
\left[ \| T_{\bm{\xi}_k} (\bm{x}_{k}) - T_{\bm{\xi}_{k-1}} (\bm{x}_{k-1}) \|   \big| \bm{\xi}_{[k-1]} \right]}_{X_k}
+ 
|\alpha_k - \alpha_{k-1}| ( \|\bm{x}_0 \| + \underbrace{\|T_{\bm{\xi}_{k-1}} (\bm{x}_{k-1})\|}_{Y_k} ).
\end{align*}
From the triangle inequality, 
\begin{align*}
X_k 
&\leq 
\underbrace{\mathbb{E}_{\bm{\xi}_k} 
\left[ \| T_{\bm{\xi}_k} (\bm{x}_{k}) - T_{\bm{\xi}_{k}} (\bm{x}_{k-1}) \|   \Big| \bm{\xi}_{[k-1]} \right]}_{X_{k,1}}
+ 
\underbrace{\mathbb{E}_{\bm{\xi}_k} 
\left[ \| T_{\bm{\xi}_{k}} (\bm{x}_{k-1}) - T (\bm{x}_{k-1})\|   \Big| \bm{\xi}_{[k-1]} \right]}_{X_{k,2}}\\
&\quad + 
\underbrace{\| T (\bm{x}_{k-1}) - T_{\bm{\xi}_{k-1}} (\bm{x}_{k-1})  \|}_{X_{k,3}}.
\end{align*}
The triangle inequality, Assumption (A1), and the mutual independence of $\bm{\xi}_k$, $\bm{x}_{k}$, and $\bm{x}_{k-1}$ ensure that 
\begin{align*}
X_{k,1} 
&= 
\mathbb{E}_{\bm{\xi}_k}
\left[
\left\|
\frac{1}{b_k} \sum_{i=1}^{b_k} (T_{\xi_{k,i}} (\bm{x}_k) - T_{\xi_{k,i}} (\bm{x}_{k-1}))
\right\| \Bigg| \bm{\xi}_{[k-1]}
\right]\\
&\leq 
\frac{1}{b_k} \sum_{i=1}^{b_k} \mathbb{E}_{\xi_{k,i}}
\left[
\| T_{\xi_{k,i}} (\bm{x}_k) - T_{\xi_{k,i}} (\bm{x}_{k-1}) \| 
\Big| \bm{\xi}_{[k-1]} \right]
\leq 
\frac{1}{b_k} \sum_{i=1}^{b_k} \|\bm{x}_k - \bm{x}_{k-1} \|
= 
\|\bm{x}_k - \bm{x}_{k-1} \|.
\end{align*} 
Proposition \ref{prop:1} and Jensen's inequality imply 
\begin{align*}
&X_{k,2} 
\leq
\sqrt{ 
\mathbb{E}_{\bm{\xi}_k} 
\left[ \| T_{\bm{\xi}_{k}} (\bm{x}_{k-1}) - T (\bm{x}_{k-1})\|^2   \Big| \bm{\xi}_{[k-1]} \right]
}
= 
\sqrt{
\mathbb{V}_{\bm{\xi}_k}\left[ T_{\bm{\xi}_{k}} (\bm{x}_{k-1}) \Big| \bm{\xi}_{[k-1]} \right]
}
\leq 
\frac{\sigma}{\sqrt{b_k}},\\
&
\mathbb{E}_{\bm{\xi}_{k-1}}
\left[ X_{k,3}  \Big| \bm{\xi}_{[k-2]} \right]
\leq 
\sqrt{
\mathbb{V}_{\bm{\xi}_{k-1}}\left[ T_{\bm{\xi}_{k-1}} (\bm{x}_{k-1}) \Big| \bm{\xi}_{[k-2]} \right]
}
\leq 
\frac{\sigma}{\sqrt{b_{k-1}}}.
\end{align*}
Accordingly, we have 
\begin{align}\label{X_k}
\mathbb{E}[ X_k ]
\leq 
\mathbb{E} \left[ \|\bm{x}_k - \bm{x}_{k-1} \| \right]
+ 
\sigma \left( \frac{1}{\sqrt{b_k}} + \frac{1}{\sqrt{b_{k-1}}} \right).
\end{align}
Moreover, Jensen's inequality and the relation $\|\bm{x} + \bm{y}\|^2 \leq 2 \|\bm{x}\|^2 + 2 \|\bm{y}\|^2$ $(\bm{x}, \bm{y} \in \mathbb{R}^d)$ ensure that, for all $\bm{x}^\star \in \mathrm{Fix}(T)$,
\begin{align*}
\mathbb{E} [Y_k] 
\leq \sqrt{\mathbb{E} \left[ \| T_{\bm{\xi}_{k-1}} (\bm{x}_{k-1}) \|^2   \right]}
=
\sqrt{
2 \mathbb{E} \left[ \| T_{\bm{\xi}_{k-1}} (\bm{x}_{k-1}) - \bm{x}^\star \|^2 \right]
+ 2 \|\bm{x}^\star\|^2
},
\end{align*}
which, together with Lemma \ref{lem:1}, implies that 
\begin{align}\label{Y_k}
\mathbb{E} [Y_k]
\leq 
\sqrt{
2 (M+\|\bm{x}^\star\|^2 + \sigma^2) 
}.
\end{align}
Therefore, from \eqref{X_k} and \eqref{Y_k},
\begin{align*}
\mathbb{E} [ \| \bm{x}_{k+1} - \bm{x}_k \| ]
&\leq
(1 - \alpha_k) \mathbb{E}[ X_k ] 
+ 
|\alpha_k - \alpha_{k-1}| (\|\bm{x}_0 \| + \mathbb{E}[ Y_k ])\\
&\leq 
(1 - \alpha_k) \mathbb{E} [ \|\bm{x}_k - \bm{x}_{k-1} \| ]
+ 
M_1 |\alpha_k - \alpha_{k-1}|
+
\sigma \left( \frac{1}{\sqrt{b_k}} + \frac{1}{\sqrt{b_{k-1}}} \right),
\end{align*}
where $M_1 \coloneqq \|\bm{x}_0 \| + \sqrt{2 (M+\|\bm{x}^\star\|^2 + \sigma^2)}$. Hence, for all $k,l \in \mathbb{N}$,
\begin{align*}
&\mathbb{E} [ \| \bm{x}_{k+l+1} - \bm{x}_{k+l} \| ]\\
&\leq
\prod_{j=l}^{k+l-1} (1- \alpha_{j+1}) \mathbb{E} [ \| \bm{x}_{l+1} - \bm{x}_{l} \| ]
+ 
M_1 \sum_{j=l}^{k+l-1} |\alpha_{j+1} - \alpha_j|
+ 
\sigma \sum_{j=l}^{k+l-1} \left( \frac{1}{\sqrt{b_{j+1}}} + \frac{1}{\sqrt{b_j}} \right).
\end{align*}
Since $\sum_{k=0}^{+\infty} \alpha_k = + \infty$ implies that $\lim_{k \to + \infty} \prod_{j=l}^{k+l-1} (1- \alpha_{j+1}) = 0$, we have that, for all $l \in \mathbb{N}$, 
\begin{align*}
\limsup_{k \to + \infty} \mathbb{E} [ \| \bm{x}_{k+1} - \bm{x}_{k} \| ]
&= 
\limsup_{k \to + \infty} \mathbb{E} [ \| \bm{x}_{k+l+1} - \bm{x}_{k+l} \| ]\\
&\leq 
M_1 \sum_{j=l}^{+ \infty} |\alpha_{j+1} - \alpha_j|
+ 
\sigma \sum_{j=l}^{+ \infty} \left( \frac{1}{\sqrt{b_{j+1}}} + \frac{1}{\sqrt{b_j}} \right).
\end{align*}
Moreover, $\sum_{k=0}^{+ \infty} | \alpha_{k+1} - \alpha_k | < + \infty$ and $\sum_{k=0}^{+ \infty} 1/\sqrt{b_k} < + \infty$ imply that $\lim_{l \to + \infty} \sum_{j=l}^{+ \infty} |\alpha_{j+1} - \alpha_j| = 0$ and $\lim_{l \to + \infty} \sum_{j=l}^{+ \infty} ( 1/\sqrt{b_{j+1}} + 1/\sqrt{b_j}) = 0$. Hence, we have 
\begin{align*}
\limsup_{k \to + \infty} \mathbb{E} [ \| \bm{x}_{k+1} - \bm{x}_{k} \| ] \leq 0, \text{ i.e., }
\lim_{k \to + \infty} \mathbb{E} [ \| \bm{x}_{k+1} - \bm{x}_{k} \| ] = 0,
\end{align*}
which completes the proof.
\end{proof}

Lemmas \ref{lem:1} and \ref{lem:2} lead to the following lemma demonstrating that when $\alpha_k$ is diminishing, the distance between $\bm{x}_k$ and $T(\bm{x}_k)$ converges to $0$ in the sense of the expected norm.

\begin{lemma}\label{lem:3}
The sequence $(\bm{x}_k)$ generated by Algorithm \ref{algo:1} under \eqref{conditions_b_a_1} and \eqref{conditions_b_a_2} for Problem \ref{prob:1} satisfies 
\begin{align*}
\mathbb{E} [ \| \bm{x}_{k+1} - T(\bm{x}_{k+1}) \| ]
\leq
\mathbb{E}[ \|\bm{x}_{k} - \bm{x}_{k+1} \|]
+
M_1 \alpha_k
+ 
\frac{\sigma}{\sqrt{b_{k}}},
\end{align*}
for all $k \in \{0\} \cup \mathbb{N}$, where $M_1$ is defined as in Lemma \ref{lem:2}, which implies that, if 
\begin{align}\label{conditions_b_a_3}
\lim_{k \to + \infty} \alpha_k = 0,
\end{align}
then
\begin{align*}
\lim_{k \to + \infty} \mathbb{E} [\| \bm{x}_k - T(\bm{x}_k) \| ] = 0.
\end{align*}
\end{lemma}

\begin{proof}
The definition of $\bm{x}_{k+1}$ (Step 4 of Algorithm \ref{algo:1}) and the triangle inequality imply that
\begin{align*}
&\mathbb{E}_{\bm{\xi}_k} \left[\| \bm{x}_{k+1} - T(\bm{x}_{k+1}) \| \Big| \bm{\xi}_{[k-1]} \right]\\
&=
\mathbb{E}_{\bm{\xi}_k} \left[\|
\alpha_k (\bm{x}_{0} - T(\bm{x}_{k+1})) 
+ (1 - \alpha_k) (T_{\bm{\xi}_{k}} (\bm{x}_{k}) - T(\bm{x}_{k+1}))
\| \Big| \bm{\xi}_{[k-1]} \right]\\
&\leq
(1 - \alpha_{k}) 
\underbrace{\mathbb{E}_{\bm{\xi}_k} \left[ 
\| T_{\bm{\xi}_{k}} (\bm{x}_{k}) - T(\bm{x}_{k+1}) \|
\Big| \bm{\xi}_{[k-1]} \right]}_{Z_k}
+ 
\alpha_{k} \underbrace{\mathbb{E}_{\bm{\xi}_k} \left[ \|\bm{x}_{0} - T(\bm{x}_{k+1})\| \Big| \bm{\xi}_{[k-1]} \right]}_{W_k}.
\end{align*} 
From the triangle inequality, 
\begin{align*}
Z_k 
\leq 
\underbrace{ 
\mathbb{E}_{\bm{\xi}_k} \left[ \| T_{\bm{\xi}_{k}} (\bm{x}_{k}) - T(\bm{x}_{k}) \| \Big| \bm{\xi}_{[k-1]} \right]
}_{Z_{k,1}}
+
\underbrace{\mathbb{E}_{\bm{\xi}_k} \left[ 
\| T(\bm{x}_{k}) - T(\bm{x}_{k+1}) \|
\Big| \bm{\xi}_{[k-1]} \right]}_{Z_{k,2}}.
\end{align*}
Proposition \ref{prop:1} and Jensen's inequality imply that
\begin{align*}
Z_{k,1} 
\leq 
\sqrt{\mathbb{E}_{\bm{\xi}_{k}} \left[ \| T_{\bm{\xi}_{k}} (\bm{x}_{k}) - T(\bm{x}_{k}) \|^2 \Big| \bm{\xi}_{[k-1]} \right]
}
= 
\sqrt{ \mathbb{V}_{\bm{\xi}_{k}} \left[ T_{\bm{\xi}_{k}} (\bm{x}_{k}) \Big| \bm{\xi}_{[k-1]} \right] }
\leq \frac{\sigma}{\sqrt{b_{k}}}.
\end{align*}
Under Assumption (A1), $T$ defined by \eqref{true_t} is nonexpansive (see also \eqref{nonexp_T}). 
Hence, we have
\begin{align*}
Z_{k,2} \leq \mathbb{E}_{\bm{\xi}_{k}} \left[ \|\bm{x}_{k} - \bm{x}_{k+1} \| \Big| \bm{\xi}_{[k-1]} \right].
\end{align*} 
Accordingly, 
\begin{align}\label{Z_k}
\mathbb{E}[Z_k]
\leq 
\mathbb{E}[Z_{k,1}]
+ 
\mathbb{E}[Z_{k,2}]
\leq
\frac{\sigma}{\sqrt{b_{k}}}
+ 
\mathbb{E}[ \|\bm{x}_{k} - \bm{x}_{k+1} \|].
\end{align}
From the triangle inequality, 
\begin{align}\label{W_k}
\begin{split}
\mathbb{E}[W_k]
&\leq 
\|\bm{x}_0\| 
+ \underbrace{\mathbb{E}[\|T(\bm{x}_{k+1}) - T(\bm{x}_k)\|]}_{\mathbb{E}[Z_{k,2}]}
+ \underbrace{\mathbb{E}[ \| T(\bm{x}_k) - T_{\bm{\xi}_k} (\bm{x}_k) \| ]}_{\mathbb{E}[Z_{k,1}]}
+ \underbrace{\mathbb{E}[ \| T_{\bm{\xi}_k} (\bm{x}_k) \|]}_{\mathbb{E}[Y_{k+1}]}\\
&\leq
M_1 + \mathbb{E}[ \|\bm{x}_{k} - \bm{x}_{k+1} \|] + \frac{\sigma}{\sqrt{b_{k}}},
\end{split}
\end{align}
where $\|\bm{x}_0 \| + \mathbb{E}[Y_{k+1}] \leq \|\bm{x}_0 \| + \sqrt{2 (M+\|\bm{x}^\star\|^2 + \sigma^2)} \eqqcolon M_1$ (see \eqref{Y_k}). Therefore, from \eqref{Z_k} and \eqref{W_k},
\begin{align*}
&\mathbb{E}[ \| \bm{x}_{k+1} - T(\bm{x}_{k+1}) \| ]\\
&\leq
(1 - \alpha_k) \left( \frac{\sigma}{\sqrt{b_{k}}}
+ 
\mathbb{E}[ \|\bm{x}_{k} - \bm{x}_{k+1} \|] \right)
+ 
\alpha_k \left( M_1 + \mathbb{E}[ \|\bm{x}_{k} - \bm{x}_{k+1} \|] + \frac{\sigma}{\sqrt{b_{k}}}  \right)\\
&=
\mathbb{E}[ \|\bm{x}_{k} - \bm{x}_{k+1} \|]
+ 
\frac{\sigma}{\sqrt{b_{k}}}
+ 
M_1 \alpha_k.
\end{align*}
Applying the conditions $\lim_{k \to + \infty} \mathbb{E}[ \|\bm{x}_{k+1} - \bm{x}_{k} \|] = 0$ in Lemma \ref{lem:2}, $\lim_{k \to + \infty} 1/\sqrt{b_k} = 0$, which comes from \eqref{conditions_b_a_2}, and $\lim_{k \to + \infty} \alpha_k = 0$ in \eqref{conditions_b_a_3}, leads to
\begin{align*}
\lim_{k \to + \infty} \mathbb{E}[ \| \bm{x}_{k+1} - T(\bm{x}_{k+1}) \| ] = 0,
\end{align*}
which completes the proof.
\end{proof}

Lemma \ref{lem:3} guarantees the convergence of Algorithm \ref{algo:1} as follows.

\begin{theorem}\label{thm:1}
The sequence $(\bm{x}_k)$ generated by Algorithm \ref{algo:1} in Problem \ref{prob:1} with $(\alpha_k)$ and $(b_k)$ satisfying 
\begin{align}\label{condition_1}
\begin{split}
&\lim_{k \to + \infty} \alpha_k = 0, \text{ } 
\sum_{k=0}^{+\infty} \alpha_k = + \infty, \text{ }
\sum_{k=0}^{+\infty} | \alpha_{k+1} - \alpha_k| < + \infty,\\
&\frac{1}{b_k} \leq \alpha_k^2 \text{ } (k \in \{0\} \cup \mathbb{N}), \text{ and } 
\sum_{k=0}^{+ \infty} \frac{1}{\sqrt{b_k}} < + \infty
\end{split}
\end{align}
converges in mean square to $P_{\mathrm{Fix}(T)}(\bm{x}_0)$; i.e., the expected squared norm $\mathbb{E}[\| \bm{x}_k - P_{\mathrm{Fix}(T)}(\bm{x}_0)\|^2]$ converges to $0$.
\end{theorem}

\begin{proof}
First, we note that \eqref{condition_1} implies that $(\alpha_k)$ and $(b_k)$ satisfy \eqref{conditions_b_a_1}, \eqref{conditions_b_a_2}, and \eqref{conditions_b_a_3}. Hence, Lemmas \ref{lem:1}, \ref{lem:2}, and \ref{lem:3} hold. Let $\bm{x}^\star \coloneqq P_{\mathrm{Fix}(T)}(\bm{x}_0) \in \mathrm{Fix}(T)$. The definition of $\bm{x}_{k+1}$ (Step 4 of Algorithm \ref{algo:1}) and the relation $\|\bm{x} + \bm{y} \|^2 = \|\bm{x}\|^2 + 2 \langle \bm{x}, \bm{y} \rangle + \|\bm{y}\|^2$ $(\bm{x},\bm{y} \in \mathbb{R}^d)$ imply that
\begin{align*}
&\mathbb{E}_{\bm{\xi}_k} \left[ \| \bm{x}_{k+1} - \bm{x}^\star \|^2 \Big| \bm{\xi}_{[k-1]} \right]\\
&= 
\mathbb{E}_{\bm{\xi}_k} \left[ \| \alpha_k (\bm{x}_{0} - \bm{x}^\star)
+ (1-\alpha_k) (T_{\bm{\xi}_k} (\bm{x}_{k}) - \bm{x}^\star) \|^2 \Big| \bm{\xi}_{[k-1]} \right]\\
&= 
\alpha_k^2 \| \bm{x}_{0} - \bm{x}^\star \|^2
+ 2 \alpha_k (1 - \alpha_k)
\mathbb{E}_{\bm{\xi}_k} \left[ \langle \bm{x}_{0} - \bm{x}^\star, T_{\bm{\xi}_k} (\bm{x}_{k}) - \bm{x}^\star \rangle \Big| \bm{\xi}_{[k-1]} \right]\\
&\quad 
+ 
(1 - \alpha_k)^2 \mathbb{E}_{\bm{\xi}_k} \left[ \|T_{\bm{\xi}_k} (\bm{x}_{k}) - \bm{x}^\star\|^2 \Big| \bm{\xi}_{[k-1]} \right],
\end{align*}
which, together with \eqref{ineq_prop}, 
\begin{align*}
\mathbb{E}_{\bm{\xi}_k} \left[ \langle \bm{x}_{0} - \bm{x}^\star, T_{\bm{\xi}_k} (\bm{x}_{k}) - \bm{x}^\star \rangle \Big| \bm{\xi}_{[k-1]} \right]
&= 
\left\langle \bm{x}_{0} - \bm{x}^\star, \mathbb{E}_{\bm{\xi}_k} \left[ T_{\bm{\xi}_k} (\bm{x}_{k}) \Big| \bm{\xi}_{[k-1]} \right] - \bm{x}^\star \right\rangle\\
&=
\langle \bm{x}_{0} - \bm{x}^\star, T (\bm{x}_{k}) - \bm{x}^\star \rangle,
\end{align*}
$\alpha_k \in (0,1)$, and $1/b_k \leq \alpha_k^2$, implies that 
\begin{align*}
&\mathbb{E}_{\bm{\xi}_k} \left[ \| \bm{x}_{k+1} - \bm{x}^\star \|^2 \Big| \bm{\xi}_{[k-1]} \right]\\
&\leq
\alpha_k^2 \| \bm{x}_{0} - \bm{x}^\star \|^2
+ 2 \alpha_k (1 - \alpha_k) \langle \bm{x}_{0} - \bm{x}^\star, T (\bm{x}_{k}) - \bm{x}^\star \rangle
+ (1 - \alpha_k)^2 \left( \| \bm{x}_{k} - \bm{x}^\star \|^2 + \frac{\sigma^2}{b_k}  \right)\\
&\leq
(1 - \alpha_k) \| \bm{x}_{k} - \bm{x}^\star \|^2 
+ 
2 \alpha_k (1 - \alpha_k) \langle \bm{x}_{0} - \bm{x}^\star, T (\bm{x}_{k}) - \bm{x}^\star \rangle
+ 
\alpha_k^2 \| \bm{x}_{0} - \bm{x}^\star \|^2
+
\alpha_k^2 \sigma^2.
\end{align*}
Accordingly, we have 
\begin{align*}
&\mathbb{E} \left[ \| \bm{x}_{k+1} - \bm{x}^\star \|^2 \right]\\
&\leq
(1 - \alpha_k) \mathbb{E} \left[ \| \bm{x}_{k} - \bm{x}^\star \|^2 \right] 
+ 
2 \alpha_k (1 - \alpha_k) 
\underbrace{\left\langle \bm{x}_{0} - \bm{x}^\star, \mathbb{E} \left[ T (\bm{x}_{k}) \right] - \bm{x}^\star \right\rangle}_{B_k}
+ 
\alpha_k \underbrace{(M_2 \alpha_k)}_{C_k},
\end{align*}
where $M_2 \coloneqq \| \bm{x}_{0} - \bm{x}^\star \|^2 + \sigma^2$. The property of the limit superior of $B_k$ ensures that there exists a subsequence $(\bm{x}_{k_i})$ of $(\bm{x}_k)$ such that 
\begin{align}\label{B_k}
\limsup_{k \to + \infty} B_k 
= \lim_{i \to + \infty} \left\langle \bm{x}_{0} - \bm{x}^\star, \mathbb{E} \left[ T (\bm{x}_{k_i}) \right] - \bm{x}^\star \right\rangle.
\end{align}
Moreover, from Lemma \ref{lem:1}, there exists a subsequence $(\bm{x}_{k_{i_j}})$ of $(\bm{x}_{k_i})$ which converges in mean. Let $\bar{\bm{x}}$ be the convergent random variable of $(\bm{x}_{k_{i_j}})$. From the triangle inequality and Assumption (A1) (which implies the nonexpansivity of $T$; see also \eqref{nonexp_T}), 
\begin{align*}
\mathbb{E}[ \| \bar{\bm{x}} - T (\bar{\bm{x}})  \|  ]
&\leq
\mathbb{E} \left[ \left\| \bar{\bm{x}} -  \bm{x}_{k_{i_j}} \right\| \right]
+ 
\mathbb{E} \left[ \left\| \bm{x}_{k_{i_j}} - T (\bm{x}_{k_{i_j}})  \right\| \right]
+ 
\mathbb{E} \left[ \left\| T (\bm{x}_{k_{i_j}}) - T (\bar{\bm{x}})  \right\| \right]\\
&\leq
2 \mathbb{E} \left[ \left\|  \bm{x}_{k_{i_j}} - \bar{\bm{x}} \right\| \right]
+ 
\mathbb{E} \left[ \left\| \bm{x}_{k_{i_j}} - T (\bm{x}_{k_{i_j}})  \right\| \right].
\end{align*}
Hence, Lemma \ref{lem:3} leads to the finding that
\begin{align*}
\mathbb{E}[ \| \bar{\bm{x}} - T (\bar{\bm{x}})  \|  ] = 0, 
\text{i.e., } \bar{\bm{x}} \in \mathrm{Fix}(T) \text{ a.s.},
\end{align*}
which, together with the closedness and convexity of $\mathrm{Fix}(T)$, implies that $\mathbb{E}[\bar{\bm{x}}] \in \mathrm{Fix}(T)$. From \eqref{B_k}, the Cauchy-Schwarz inequality, and the triangle inequality, 
\begin{align*}
\limsup_{k \to + \infty} B_k 
&= \lim_{j \to + \infty} \left\langle \bm{x}_{0} - \bm{x}^\star, \mathbb{E} \left[ T (\bm{x}_{k_{i_j}}) \right] - \bm{x}^\star \right\rangle\\
&= 
\lim_{j \to + \infty} \left\langle \bm{x}_{0} - \bm{x}^\star, \mathbb{E} \left[ T (\bm{x}_{k_{i_j}}) - \bm{x}_{k_{i_j}} \right] 
+ \mathbb{E} \left[ \bm{x}_{k_{i_j}} - \bar{\bm{x}} \right] + \mathbb{E}[\bar{\bm{x}}] - \bm{x}^\star \right\rangle\\
&\leq
\| \bm{x}_{0} - \bm{x}^\star \| \lim_{j \to + \infty}
\left\| \mathbb{E} \left[ T (\bm{x}_{k_{i_j}}) - \bm{x}_{k_{i_j}} \right] \right\| 
+ \left\| \mathbb{E} \left[ \bm{x}_{k_{i_j}} - \bar{\bm{x}} \right] \right\|\\
&\quad 
+ 
\left\langle \bm{x}_{0} - \bm{x}^\star,\mathbb{E}[\bar{\bm{x}}] - \bm{x}^\star \right\rangle,
\end{align*}
which, together with Jensen's inequality, Lemma \ref{lem:3}, $\lim_{j \to + \infty} \mathbb{E}[\| \bm{x}_{k_{i_j}} - \bar{\bm{x}} \|] = 0$, and the property of $P_{\mathrm{Fix}(T)}$ (i.e., $\langle \bm{x}_{0} - P_{\mathrm{Fix}(T)}(\bm{x}_0), \bm{x} - P_{\mathrm{Fix}(T)}(\bm{x}_0) \rangle \leq 0$ for all $\bm{x} \in \mathrm{Fix}(T)$), implies that 
\begin{align*}
\limsup_{k \to + \infty} B_k
\leq 
\left\langle \bm{x}_{0} - \bm{x}^\star,\mathbb{E}[\bar{\bm{x}}] - \bm{x}^\star \right\rangle
\leq 0.
\end{align*}
Hence, together with $\lim_{k \to + \infty} \alpha_k = 0$, 
for all $\epsilon > 0$, there exists $k_0 \in \mathbb{N}$ such that 
\begin{align*}
B_k \leq \epsilon \text{ and } C_k \leq \epsilon
\end{align*}
for all $k \geq k_0$. Therefore, for all $k \geq k_0$, 
\begin{align*}
\mathbb{E} \left[ \| \bm{x}_{k+1} - \bm{x}^\star \|^2 \right]
&\leq
(1 - \alpha_k) \mathbb{E} \left[ \| \bm{x}_{k} - \bm{x}^\star \|^2 \right] 
+ 
2 \alpha_k (1 - \alpha_k) \epsilon
+ 
\alpha_k \epsilon\\
&\leq
(1 - \alpha_k) \mathbb{E} \left[ \| \bm{x}_{k} - \bm{x}^\star \|^2 \right]
+ 3 \alpha_k \epsilon\\
&= 
(1 - \alpha_k) \mathbb{E} \left[ \| \bm{x}_{k} - \bm{x}^\star \|^2 \right]
+ 3 \{ 1 - (1 - \alpha_k) \} \epsilon,
\end{align*}
which implies that 
\begin{align*}
\mathbb{E} \left[ \| \bm{x}_{k+1} - \bm{x}^\star \|^2 \right]
\leq
\prod_{j=k_0}^k (1 - \alpha_j) \mathbb{E} \left[ \| \bm{x}_{k_0} - \bm{x}^\star \|^2 \right]
+ 3 \epsilon \left\{ 1 -  \prod_{j=k_0}^k (1 - \alpha_j)  \right\}.
\end{align*}
Since $\sum_{k=0}^{+\infty} \alpha_k = + \infty$ implies that $\prod_{j=k_0}^{+\infty} (1 - \alpha_j) = 0$, we have that, for all $\epsilon > 0$, 
\begin{align*}
\limsup_{k \to + \infty} \mathbb{E} \left[ \| \bm{x}_{k+1} - \bm{x}^\star \|^2 \right]
\leq 3 \epsilon, \text{ i.e., }
\lim_{k \to + \infty} \mathbb{E} \left[ \| \bm{x}_{k+1} - \bm{x}^\star \|^2 \right] = 0.
\end{align*}
This completes the proof.
\end{proof}

Let us examine some examples that satisfy \eqref{condition_1}. A diminishing step size $\alpha_k$ defined for all $k \in \{0\} \cup \mathbb{N}$ by 
\begin{align}\label{diminishing_1}
\alpha_k = \frac{1}{(k+1)^a} \in (0,1],
\end{align}
where $a \in (0,1]$, satisfies $\lim_{k \to + \infty} \alpha_k = 0$. Moreover, $\sum_{k=0}^{+ \infty} \alpha_k = + \infty$ holds, from
\begin{align}\label{diminishing}
\sum_{k=0}^{K-1} \alpha_k
\geq 
\begin{dcases}
\frac{(K+1)^{1-a} -1}{1-a} 
&\text{ } \left(a \in (0, 1) \right)\\
\log (K+1)  
&\text{ } \left(a =1 \right).
\end{dcases}
\end{align}
From $\sum_{k=0}^{K-1} |\alpha_{k+1} - \alpha_k|= \alpha_0 - \alpha_K \leq \alpha_0 = 1$, we have $\sum_{k=0}^{+ \infty} |\alpha_{k+1} - \alpha_k| < + \infty$.

An increasing batch size $b_k$ such as 
\begin{align}\label{increasing}
\begin{split}
&\text{[Polynomial increasing batch size] }
b_k = (a_0 k + b_0)^c, \text{ or}\\
&\text{[Exponential increasing batch size] }
b_k = b_0 \delta^k,
\end{split}
\end{align} 
where $a_0 > 0$, $c > 2$, and $\delta > 1$, satisfies $\sum_{k=0}^{+ \infty} 1/\sqrt{b_k} < + \infty$, from
\begin{align}\label{B_value}
\sum_{k=0}^{K-1} \frac{1}{\sqrt{b_k}}
\leq
B 
\coloneqq
\begin{dcases}
\frac{c}{(c-2) \min \{ a_0, b_0 \}^{\frac{c}{2}}} &\text{ (Polynomial increasing batch size)}\\
\frac{\sqrt{\delta}}{(\sqrt{\delta} -1) \sqrt{b_0}} &\text{ (Exponential increasing batch size)}. 
\end{dcases}
\end{align}
Moreover, $\alpha_k$ defined by \eqref{diminishing_1} and $b_k$ defined by \eqref{increasing} satisfy$1/b_k \leq \alpha_k^2$ for a sufficiently large $k$.

A constant batch size $b_k = b$ does not satisfy $\sum_{k=0}^{+ \infty} 1/\sqrt{b_k} < + \infty$. Hence, Theorem \ref{thm:1} says that the use of increasing batch sizes is essential for guaranteeing the convergence of Algorithm \ref{algo:1}. This claim also appears in the convergence analysis of mini-batch stochastic gradient descent \cite{umeda2025increasing}.

\subsection{Convergence rate}\label{sec:3.2}
The following mapping $T_{\bm{\xi}_k}^\lambda$ is a convex combination of $\mathrm{Id}$ and $T_{\bm{\xi}_k}$ defined by \eqref{mini_batch_t}:
\begin{align}\label{S_k}
T_{\bm{\xi}_k}^\lambda \coloneqq \lambda \mathrm{Id} + (1-\lambda)T_{\bm{\xi}_k}, 
\end{align}
where $\lambda \in [0,1)$. 
We will consider the following algorithm:
\begin{align}\label{modified_H}
\bm{x}_{k+1} 
\coloneqq T_{\bm{\xi}_k}^\lambda (\bm{x}_k) + \alpha_k \left(\bm{x}_0 - T_{\bm{\xi}_k}^\lambda (\bm{x}_k) \right)
= 
T_{\bm{\xi}_k}^\lambda (\bm{x}_k) - \alpha_k \nabla f_0 \left( T_{\bm{\xi}_k}^\lambda (\bm{x}_k) \right),
\end{align}
where $f_0 (\bm{x}) \coloneqq (1/2)\|\bm{x} - \bm{x}_0\|^2$ satisfies $\nabla f_0 (\bm{x}) = \bm{x} - \bm{x}_0$. Under Assumptions (A1), (A2), and (A3), the following hold:
\begin{enumerate}
\item[(A1)'] $T_i^\lambda \coloneqq \lambda \mathrm{Id} + (1 - \lambda) T_i$ $(i\in [n])$ is nonexpansive;
\item[(A2)'] $\mathbb{V}_\xi [T_\xi^\lambda (\bm{x})] = (1 - \lambda)^2 \mathbb{V}_\xi [T_\xi (\bm{x})] \leq (1 - \lambda)^2 \sigma^2 \leq \sigma^2$;
\item[(A3)'] $\mathrm{Fix}(T^\lambda) = \mathrm{Fix}(T) \neq \emptyset$, 
where $T^\lambda \coloneqq \lambda \mathrm{Id} + (1 - \lambda) T$.
\end{enumerate} 
Since algorithm \eqref{modified_H} can be obtained by replacing $T_{\bm{\xi}_k}$ in Algorithm \ref{algo:1} with $T_{\bm{\xi}_k}^\lambda$, Theorem \ref{thm:1} guarantees that, under (A1)', (A2)', and (A3)', algorithm \eqref{modified_H} with any $\lambda \in [0,1]$ converges in mean square to $\bm{x}^\star = P_{\mathrm{Fix}(T^\lambda)}(\bm{x}_0) = P_{\mathrm{Fix}(T)}(\bm{x}_0)$.

To prove the following lemma, we need to restrict $\lambda$ to $1/2 < \lambda \leq 3/4$ so that $\alpha_k$ can satisfy $0 < \alpha_k \leq 1$. 

\begin{lemma}\label{lem:4}
The sequence $(\bm{x}_k)$ generated by \eqref{modified_H} for Problem \ref{prob:1} with $\alpha_k \leq (2 \lambda -1)/(2(1-\lambda))$ $(k \in \{0\} \cup \mathbb{N})$ satisfies 
\begin{align*}
&\mathbb{E}_{\bm{\xi}_k} \left[ \| \bm{x}_{k+1} - \bm{x}^\star \|^2 \Big| \bm{\xi}_{[k-1]} \right]\\
&\leq 
\| \bm{x}_{k} - \bm{x}^\star \|^2
+ 
2 \alpha_k \left( f_0^\star - f_0 (\bm{x}_k)  \right)
+ 
2 \alpha_k^2 \mathbb{E}_{\bm{\xi}_k} \left[ \left\| \nabla f_0 \left( T_{\bm{\xi}_k}^\lambda (\bm{x}_k) \right)  \right\|^2 \Big| \bm{\xi}_{[k-1]} \right] + \frac{\sigma^2}{b_k},
\end{align*}
for all $k \in \{0\} \cup \mathbb{N}$, where $1/2 < \lambda \leq 3/4$ and $f_0^\star \coloneqq f_0 (P_{\mathrm{Fix}(T)} (\bm{x}_0))$ is the optimal value of $f_0$ over $\mathrm{Fix}(T)$.
\end{lemma}

\begin{proof}
Let $\bm{x}^\star = P_{\mathrm{Fix}(T)}(\bm{x}_0)$. From \eqref{modified_H} and the relation $2 \langle \bm{x}, \bm{y} \rangle = \|\bm{x} + \bm{y}\|^2 - \|\bm{x}\|^2 - \|\bm{y}\|^2 $ $(\bm{x},\bm{y} \in \mathbb{R}^d)$, we have 
\begin{align}\label{S_1}
&2 \mathbb{E}_{\bm{\xi}_k} \left[ \left\langle T_{\bm{\xi}_k}^\lambda (\bm{x}_k) - \bm{x}_k, \bm{x}_k - \bm{x}^\star \right\rangle \Big| \bm{\xi}_{[k-1]} \right]\nonumber \\
&= 
2 \mathbb{E}_{\bm{\xi}_k} \left[ \left\langle \bm{x}_{k+1} - \bm{x}_k, \bm{x}_k - \bm{x}^\star \right\rangle \Big| \bm{\xi}_{[k-1]} \right]
+ 
2 \alpha_k \mathbb{E}_{\bm{\xi}_k} \left[ \left\langle \nabla f_0 \left( T_{\bm{\xi}_k}^\lambda (\bm{x}_k) \right), \bm{x}_k - \bm{x}^\star \right\rangle \Big| \bm{\xi}_{[k-1]} \right]\nonumber \\
&= 
\mathbb{E}_{\bm{\xi}_k} \left[ \| \bm{x}_{k+1} - \bm{x}^\star \|^2 \Big| \bm{\xi}_{[k-1]} \right]
- 
\mathbb{E}_{\bm{\xi}_k} \left[ \| \bm{x}_{k+1} - \bm{x}_k \|^2 \Big| \bm{\xi}_{[k-1]} \right]
- \| \bm{x}_{k} - \bm{x}^\star \|^2 \nonumber \\
&\quad 
+ 
2 \alpha_k \mathbb{E}_{\bm{\xi}_k} \left[ \left\langle \nabla f_0 \left( T_{\bm{\xi}_k}^\lambda (\bm{x}_k) \right), \bm{x}_k - \bm{x}^\star \right\rangle \Big| \bm{\xi}_{[k-1]} \right].
\end{align}
Meanwhile, from \eqref{S_k},
\begin{align*}
&2 \mathbb{E}_{\bm{\xi}_k} \left[ \left\langle  \bm{x}_k - T_{\bm{\xi}_k}^\lambda (\bm{x}_k), \bm{x}_k - \bm{x}^\star \right\rangle \Big| \bm{\xi}_{[k-1]} \right]
= 2 (1 - \lambda) \mathbb{E}_{\bm{\xi}_k} \left[ \left\langle  \bm{x}_k - T_{\bm{\xi}_k} (\bm{x}_k), \bm{x}_k - \bm{x}^\star \right\rangle \Big| \bm{\xi}_{[k-1]} \right].
\end{align*} 
From the relation $2 \langle \bm{x}, \bm{y} \rangle = \|\bm{x}\|^2 + \|\bm{y}\|^2 - \|\bm{x} - \bm{y}\|^2$ $(\bm{x},\bm{y} \in \mathbb{R}^d)$, we have
\begin{align*}
&2 \mathbb{E}_{\bm{\xi}_k} \left[ \left\langle  \bm{x}_k - T_{\bm{\xi}_k} (\bm{x}_k), \bm{x}_k - \bm{x}^\star \right\rangle \Big| \bm{\xi}_{[k-1]} \right]\\
&= 
\mathbb{E}_{\bm{\xi}_k} \left[ \left\| T_{\bm{\xi}_k} (\bm{x}_k) - \bm{x}_k \right\|^2 \Big| \bm{\xi}_{[k-1]} \right]
+ 
\| \bm{x}_{k} - \bm{x}^\star \|^2
- 
\mathbb{E}_{\bm{\xi}_k} \left[ \left\| T_{\bm{\xi}_k} (\bm{x}_k) - \bm{x}^\star \right\|^2 \Big| \bm{\xi}_{[k-1]} \right].
\end{align*}
Moreover, from \eqref{ineq_prop}, we also have 
\begin{align*}
\mathbb{E}_{\bm{\xi}_k} \left[ \left\| T_{\bm{\xi}_k} (\bm{x}_k) - \bm{x}^\star \right\|^2 \Big| \bm{\xi}_{[k-1]} \right]
\leq
\| \bm{x}_{k} - \bm{x}^\star \|^2
+ 
\frac{\sigma^2}{b_k},
\end{align*}
which implies that 
\begin{align}\label{S_2}
2 \mathbb{E}_{\bm{\xi}_k} \left[ \left\langle  \bm{x}_k - T_{\bm{\xi}_k}^\lambda (\bm{x}_k), \bm{x}_k - \bm{x}^\star \right\rangle \Big| \bm{\xi}_{[k-1]} \right]
\geq 
(1 - \lambda) 
\left\{ 
\mathbb{E}_{\bm{\xi}_k} \left[ \left\| T_{\bm{\xi}_k} (\bm{x}_k) - \bm{x}_k \right\|^2 \Big| \bm{\xi}_{[k-1]} \right]
- \frac{\sigma^2}{b_k}
\right\}.
\end{align}
Accordingly, from \eqref{S_1} and \eqref{S_2},
\begin{align*}
&- (1 - \lambda) \mathbb{E}_{\bm{\xi}_k} \left[ \left\| T_{\bm{\xi}_k} (\bm{x}_k) - \bm{x}_k \right\|^2 \Big| \bm{\xi}_{[k-1]} \right]
+ \frac{(1 - \lambda) \sigma^2}{b_k} \\
&\quad\geq 
\mathbb{E}_{\bm{\xi}_k} \left[ \| \bm{x}_{k+1} - \bm{x}^\star \|^2 \Big| \bm{\xi}_{[k-1]} \right]
- 
\mathbb{E}_{\bm{\xi}_k} \left[ \| \bm{x}_{k+1} - \bm{x}_k \|^2 \Big| \bm{\xi}_{[k-1]} \right]
- \| \bm{x}_{k} - \bm{x}^\star \|^2\\
&\quad 
+ 
2 \alpha_k \mathbb{E}_{\bm{\xi}_k} \left[ \left\langle \nabla f_0 \left( T_{\bm{\xi}_k}^\lambda (\bm{x}_k) \right), \bm{x}_k - \bm{x}^\star \right\rangle \Big| \bm{\xi}_{[k-1]} \right],
\end{align*}
which implies that 
\begin{align*}
&\mathbb{E}_{\bm{\xi}_k} \left[ \| \bm{x}_{k+1} - \bm{x}^\star \|^2 \Big| \bm{\xi}_{[k-1]} \right]\\
&\leq 
\| \bm{x}_{k} - \bm{x}^\star \|^2
+ 
\mathbb{E}_{\bm{\xi}_k} \left[ \| \bm{x}_{k+1} - \bm{x}_k \|^2 \Big| \bm{\xi}_{[k-1]} \right]
- 
2 \alpha_k \mathbb{E}_{\bm{\xi}_k} \left[ \left\langle \nabla f_0 \left( T_{\bm{\xi}_k}^\lambda (\bm{x}_k) \right), \bm{x}_k - \bm{x}^\star \right\rangle \Big| \bm{\xi}_{[k-1]} \right]\\
&\quad 
- (1 - \lambda) \mathbb{E}_{\bm{\xi}_k} \left[ \left\| T_{\bm{\xi}_k} (\bm{x}_k) - \bm{x}_k \right\|^2 \Big| \bm{\xi}_{[k-1]} \right]
+ 
\frac{\sigma^2}{b_k}.
\end{align*}
Since the relation $\|\bm{x} + \bm{y}\|^2 \leq 2 \|\bm{x}\|^2 + 2 \|\bm{y}\|^2$ $(\bm{x},\bm{y} \in \mathbb{R}^d)$, \eqref{S_k}, and \eqref{modified_H} lead to 
\begin{align*}
&\mathbb{E}_{\bm{\xi}_k} \left[ \| \bm{x}_{k+1} - \bm{x}_k \|^2 \Big| \bm{\xi}_{[k-1]} \right]\\
&\leq 
2 \mathbb{E}_{\bm{\xi}_k} \left[ \left\| \bm{x}_{k+1} - T_{\bm{\xi}_k}^\lambda (\bm{x}_k) \right\|^2 \Big| \bm{\xi}_{[k-1]} \right]
+ 
2 \mathbb{E}_{\bm{\xi}_k} \left[ \left\| T_{\bm{\xi}_k}^\lambda (\bm{x}_k) - \bm{x}_k \right\|^2 \Big| \bm{\xi}_{[k-1]} \right]\\
&= 
2 \alpha_k^2 \mathbb{E}_{\bm{\xi}_k} \left[ \left\| \nabla f_0 \left( T_{\bm{\xi}_k}^\lambda (\bm{x}_k) \right)  \right\|^2 \Big| \bm{\xi}_{[k-1]} \right]
+ 
2 (1-\lambda)^2 \mathbb{E}_{\bm{\xi}_k} \left[ \left\| T_{\bm{\xi}_k} (\bm{x}_k) - \bm{x}_k \right\|^2 \Big| \bm{\xi}_{[k-1]} \right],
\end{align*}
we have 
\begin{align*}
&\mathbb{E}_{\bm{\xi}_k} \left[ \| \bm{x}_{k+1} - \bm{x}^\star \|^2 \Big| \bm{\xi}_{[k-1]} \right]\\
&\leq 
\| \bm{x}_{k} - \bm{x}^\star \|^2
+ 
2 \alpha_k \underbrace{\mathbb{E}_{\bm{\xi}_k} \left[ \left\langle \nabla f_0 \left( T_{\bm{\xi}_k}^\lambda (\bm{x}_k) \right), \bm{x}^\star - \bm{x}_k \right\rangle \Big| \bm{\xi}_{[k-1]} \right]}_{G_k} + \frac{\sigma^2}{b_k}\\
&\quad 
+ 
2 \alpha_k^2 \mathbb{E}_{\bm{\xi}_k} \left[ \left\| \nabla f_0 \left( T_{\bm{\xi}_k}^\lambda (\bm{x}_k) \right)  \right\|^2 \Big| \bm{\xi}_{[k-1]} \right] 
+ (1 - \lambda) (1 - 2 \lambda)  \mathbb{E}_{\bm{\xi}_k} \left[ \left\| T_{\bm{\xi}_k} (\bm{x}_k) - \bm{x}_k \right\|^2 \Big| \bm{\xi}_{[k-1]} \right].
\end{align*}
The convexity of $f_0$ ensures that 
\begin{align*}
G_k
&=
\mathbb{E}_{\bm{\xi}_k} \left[ \left\langle \nabla f_0 \left( T_{\bm{\xi}_k}^\lambda (\bm{x}_k) \right), \bm{x}^\star - T_{\bm{\xi}_k}^\lambda (\bm{x}_k)  \right\rangle \Big| \bm{\xi}_{[k-1]} \right]\\
&\quad +
\mathbb{E}_{\bm{\xi}_k} \left[ \left\langle \nabla f_0 \left( T_{\bm{\xi}_k}^\lambda (\bm{x}_k) \right),  T_{\bm{\xi}_k}^\lambda (\bm{x}_k) - \bm{x}_k \right\rangle \Big| \bm{\xi}_{[k-1]} \right]\\
&\leq 
f_0^\star - f_0 \left( T_{\bm{\xi}_k}^\lambda (\bm{x}_k) \right)
+ 
\mathbb{E}_{\bm{\xi}_k} \left[ \left\langle \nabla f_0 \left( T_{\bm{\xi}_k}^\lambda (\bm{x}_k) \right),  T_{\bm{\xi}_k}^\lambda (\bm{x}_k) - \bm{x}_k \right\rangle \Big| \bm{\xi}_{[k-1]} \right]\\
&= 
f_0^\star - f_0 (\bm{x}_k) + f_0 (\bm{x}_k) - f_0 \left( T_{\bm{\xi}_k}^\lambda (\bm{x}_k) \right)
+ 
\mathbb{E}_{\bm{\xi}_k} \left[ \left\langle \nabla f_0 \left( T_{\bm{\xi}_k}^\lambda (\bm{x}_k) \right),  T_{\bm{\xi}_k}^\lambda (\bm{x}_k) - \bm{x}_k \right\rangle \Big| \bm{\xi}_{[k-1]} \right]\\
&\leq 
f_0^\star - f_0 (\bm{x}_k) 
-
\mathbb{E}_{\bm{\xi}_k} \left[ \left\langle \nabla f_0 ( \bm{x}_k),  T_{\bm{\xi}_k}^\lambda (\bm{x}_k) - \bm{x}_k \right\rangle \Big| \bm{\xi}_{[k-1]} \right]\\
&\quad 
+ 
\mathbb{E}_{\bm{\xi}_k} \left[ \left\langle \nabla f_0 \left( T_{\bm{\xi}_k}^\lambda (\bm{x}_k) \right),  T_{\bm{\xi}_k}^\lambda (\bm{x}_k) - \bm{x}_k \right\rangle \Big| \bm{\xi}_{[k-1]} \right]\\
&= 
f_0^\star - f_0 (\bm{x}_k)
+ 
\underbrace{\mathbb{E}_{\bm{\xi}_k} \left[ \left\langle \nabla f_0 \left( T_{\bm{\xi}_k}^\lambda (\bm{x}_k) \right) - \nabla f_0 ( \bm{x}_k),  T_{\bm{\xi}_k}^\lambda (\bm{x}_k) - \bm{x}_k \right\rangle \Big| \bm{\xi}_{[k-1]} \right]}_{G_{k,1}}.
\end{align*}
The Cauchy-Schwarz inequality, the $1$-Lipschitz continuity of $\nabla f_0$, and \eqref{S_k} ensure that
\begin{align*}
G_{k,1} 
&\leq 
\mathbb{E}_{\bm{\xi}_k} \left[ \left\| \nabla f_0 \left( T_{\bm{\xi}_k}^\lambda (\bm{x}_k) \right) - \nabla f_0 ( \bm{x}_k) \right\| \left\| T_{\bm{\xi}_k}^\lambda (\bm{x}_k) - \bm{x}_k \right\| \Big| \bm{\xi}_{[k-1]} \right]\\
&\leq
\mathbb{E}_{\bm{\xi}_k} \left[  \left\| T_{\bm{\xi}_k}^\lambda (\bm{x}_k) - \bm{x}_k \right\|^2 \Big| \bm{\xi}_{[k-1]} \right]
= 
(1 - \lambda)^2 \mathbb{E}_{\bm{\xi}_k} \left[  \left\| T_{\bm{\xi}_k} (\bm{x}_k) - \bm{x}_k \right\|^2 \Big| \bm{\xi}_{[k-1]} \right].
\end{align*}
Hence, 
\begin{align*}
&\mathbb{E}_{\bm{\xi}_k} \left[ \| \bm{x}_{k+1} - \bm{x}^\star \|^2 \Big| \bm{\xi}_{[k-1]} \right]\\
&\leq 
\| \bm{x}_{k} - \bm{x}^\star \|^2
+ 
2 \alpha_k \left\{   
f_0^\star - f_0 (\bm{x}_k)
+ 
(1 - \lambda)^2 \mathbb{E}_{\bm{\xi}_k} \left[  \left\| T_{\bm{\xi}_k} (\bm{x}_k) - \bm{x}_k \right\|^2 \Big| \bm{\xi}_{[k-1]} \right]
\right\} + \frac{\sigma^2}{b_k}\\
&\quad 
+ 
2 \alpha_k^2 \mathbb{E}_{\bm{\xi}_k} \left[ \left\| \nabla f_0 \left( T_{\bm{\xi}_k}^\lambda (\bm{x}_k) \right)  \right\|^2 \Big| \bm{\xi}_{[k-1]} \right] 
+ (1 - \lambda) (1 - 2 \lambda)  \mathbb{E}_{\bm{\xi}_k} \left[ \left\| T_{\bm{\xi}_k} (\bm{x}_k) - \bm{x}_k \right\|^2 \Big| \bm{\xi}_{[k-1]} \right]\\
&= 
\| \bm{x}_{k} - \bm{x}^\star \|^2
+ 
2 \alpha_k \left( f_0^\star - f_0 (\bm{x}_k)  \right)
+ 
2 \alpha_k^2 \mathbb{E}_{\bm{\xi}_k} \left[ \left\| \nabla f_0 \left( T_{\bm{\xi}_k}^\lambda (\bm{x}_k) \right)  \right\|^2 \Big| \bm{\xi}_{[k-1]} \right] + \frac{\sigma^2}{b_k}\\
&\quad 
+ 
(1 - \lambda) \left\{ 2 (1 - \lambda) \alpha_k + (1 - 2 \lambda) \right\}
\mathbb{E}_{\bm{\xi}_k} \left[ \left\| T_{\bm{\xi}_k} (\bm{x}_k) - \bm{x}_k \right\|^2 \Big| \bm{\xi}_{[k-1]} \right],
\end{align*}
which, together with $\alpha_k \leq (2 \lambda - 1)/(2(1-\lambda))$, leads to the assertion in Lemma \ref{lem:4}.
\end{proof}

Lemma \ref{lem:4} enables us to conduct the following convergence rate analysis on algorithm \eqref{modified_H}.

\begin{theorem}\label{thm:2}
Let $(\bm{x}_k)$ be the sequence generated by \eqref{modified_H} for Problem \ref{prob:1} with $(\alpha_k)$ and $(b_k)$ satisfying 
\begin{align*}
\frac{1}{b_k} \leq \alpha_k \leq \frac{2 \lambda - 1}{2 (1-\lambda)} \text{ } (k\in \{0\} \cup \mathbb{N}), 
\text{ }
\sum_{k=0}^{+\infty} \alpha_k = + \infty,
\text{ and }
\sum_{k=0}^{+\infty} \frac{1}{b_k} < + \infty.
\end{align*}
Then, for all $K \in \mathbb{N}$,
\begin{align*}
\min_{k \in [0:K-1]} \mathbb{E} \left[ f_0 (\bm{x}_k) \right]
&\leq
\frac{1}{\sum_{k=0}^{K-1} \alpha_k} \sum_{k=0}^{K-1} \alpha_k \mathbb{E}[ f_0 (\bm{x}_k)]\\
&\leq
f_0^\star
+
\frac{1}{2 \sum_{k=0}^{K-1} \alpha_k}
\left( \| \bm{x}_0 - \bm{x}^\star \|^2
+
M_3 \sum_{k=0}^{K-1} \alpha_k^2
+
\sigma^2 \sum_{k=0}^{K-1} \frac{1}{b_k}
\right)\\
&= 
f_0^\star 
+
\begin{dcases}
O \left( \frac{\sum_{k=0}^{K-1} \alpha_k^2}{\sum_{k=0}^{K-1} \alpha_k} \right) 
&\text{ } \left(  \sum_{k=0}^{K-1} \alpha_k^2 = o \left( \sum_{k=0}^{K-1} \alpha_k  \right) \right)\\
O \left( \frac{1}{\sum_{k=0}^{K-1} \alpha_k} \right)
&\text{ } \left( \sum_{k=0}^{+ \infty} \alpha_k^2 < + \infty \right), 
\end{dcases}
\end{align*}
where $M_3 \coloneqq 4 \{(M + \sigma^2) + \| \nabla f_0 (\bm{x}^\star)\|^2\}$ and $M$ is defined as in Lemma \ref{lem:1}.
\end{theorem}

\begin{proof}
Let $\bm{x}^\star = P_{\mathrm{Fix}(T)} (\bm{x}_0)$. From the relation $\|\bm{x} + \bm{y}\|^2 \leq 2 \|\bm{x}\|^2 + 2 \|\bm{y}\|^2$ $(\bm{x},\bm{y} \in \mathbb{R}^d)$, we have
\begin{align*}
&\mathbb{E}_{\bm{\xi}_k} \left[ \left\| \nabla f_0 \left( T_{\bm{\xi}_k}^\lambda (\bm{x}_k) \right)  \right\|^2 \Big| \bm{\xi}_{[k-1]} \right]\\
&\leq 
2 \mathbb{E}_{\bm{\xi}_k} \left[ \left\| \nabla f_0 \left( T_{\bm{\xi}_k}^\lambda (\bm{x}_k) \right) - \nabla f_0 (\bm{x}^\star)  \right\|^2 \Big| \bm{\xi}_{[k-1]} \right]
+ 
2 \left\| \nabla f_0 (\bm{x}^\star) \right\|^2,
\end{align*}
which, together with the $1$-Lipschitz continuity of $\nabla f_0$, implies that
\begin{align*}
&\mathbb{E}_{\bm{\xi}_k} \left[ \left\| \nabla f_0 \left( T_{\bm{\xi}_k}^\lambda (\bm{x}_k) \right)  \right\|^2 \Big| \bm{\xi}_{[k-1]} \right]
\leq 
2 \mathbb{E}_{\bm{\xi}_k} \left[ \left\|  T_{\bm{\xi}_k}^\lambda (\bm{x}_k)  - \bm{x}^\star  \right\|^2 \Big| \bm{\xi}_{[k-1]} \right]
+ 
2 \left\| \nabla f_0 (\bm{x}^\star) \right\|^2.
\end{align*}
A similar argument to the one proving Lemma \ref{lem:1} leads to the finding that algorithm \eqref{modified_H} with $1/b_k \leq \alpha_k$ $(k\in \{0\} \cup \mathbb{N})$ satisfies 
\begin{align*}
\mathbb{E} \left[ \left\|  T_{\bm{\xi}_k}^\lambda (\bm{x}_k)  - \bm{x}^\star  \right\|^2   \right]
\leq M + \sigma^2.
\end{align*}
Accordingly, we have 
\begin{align*}
\mathbb{E}_{\bm{\xi}_k} \left[ \left\| \nabla f_0 \left( T_{\bm{\xi}_k}^\lambda (\bm{x}_k) \right)  \right\|^2 \Big| \bm{\xi}_{[k-1]} \right]
\leq 
2 \left(M + \sigma^2 \right) + 2 \left\| \nabla f_0 (\bm{x}^\star) \right\|^2
\eqqcolon \frac{M_3}{2}.
\end{align*}
Lemma \ref{lem:4} thus ensures that, for all $k \in \{0\} \cup \mathbb{N}$, 
\begin{align*}
2 \alpha_k \mathbb{E}[ f_0 (\bm{x}_k) - f_0^\star ]
\leq 
\mathbb{E} \left[ \| \bm{x}_k - \bm{x}^\star \|^2  \right]
- 
\mathbb{E} \left[ \| \bm{x}_{k+1} - \bm{x}^\star \|^2  \right]
+ 
M_3 \alpha_k^2 + \frac{\sigma^2}{b_k}.
\end{align*}
Let $K \in \mathbb{N}$. Summing the above inequality from $k=0$ to $K-1$ guarantees that
\begin{align*}
2 \sum_{k=0}^{K-1} \alpha_k \mathbb{E}[ f_0 (\bm{x}_k)] - f_0^\star
\leq
\| \bm{x}_0 - \bm{x}^\star \|^2 
- 
\mathbb{E} \left[ \| \bm{x}_{K} - \bm{x}^\star \|^2  \right]
+ 
M_3 \sum_{k=0}^{K-1} \alpha_k^2
+
\sigma^2 \sum_{k=0}^{K-1} \frac{1}{b_k},
\end{align*}
which, together with $\mathbb{E} [ \| \bm{x}_{K} - \bm{x}^\star \|^2 ] \geq 0$, implies that
\begin{align*}
\min_{k \in [0:K-1]} \mathbb{E}[ f_0 (\bm{x}_k)]
&\leq
\frac{1}{\sum_{k=0}^{K-1} \alpha_k} \sum_{k=0}^{K-1} \alpha_k \mathbb{E}[ f_0 (\bm{x}_k)]\\
&\leq
f_0^\star 
+
\frac{\| \bm{x}_0 - \bm{x}^\star \|^2}{2 \sum_{k=0}^{K-1} \alpha_k}
+
\frac{M_3 \sum_{k=0}^{K-1} \alpha_k^2}{2 \sum_{k=0}^{K-1} \alpha_k}
+
\frac{\sigma^2 \sum_{k=0}^{K-1} b_k^{-1}}{2 \sum_{k=0}^{K-1} \alpha_k}.
\end{align*}
This completes the proof.
\end{proof}

\subsubsection{Concrete convergence rates}
\label{sec:3.2.1}
The diminishing step size $\alpha_k$ defined by \eqref{diminishing_1} may be modified to 
\begin{align}\label{modifined_a}
\alpha_k = \frac{2 \lambda - 1}{2(1-\lambda) (k+1)^a} \leq \frac{2 \lambda - 1}{2(1-\lambda)},
\end{align}
where $a \in (0,1)$. From \eqref{diminishing}, $\sum_{k=0}^{+\infty} \alpha_k = + \infty$ holds. Using an increasing batch size $b_k$ defined by \eqref{increasing} satisfies $1/b_k \leq \alpha_k$ for a sufficiently large $k$. Moreover, 
\begin{align}\label{diminishing_0}
\sum_{k=0}^{K-1} \alpha_k^2
\leq
\frac{(2 \lambda - 1)^2}{2^2 (1-\lambda)^2}
\times
\begin{dcases}
\frac{K^{1 - 2a}}{1 - 2a} &\text{ } \left( a \in \left(0,\frac{1}{2} \right) \right)\\
1 + \log K  &\text{ } \left( a  = \frac{1}{2} \right)\\
\frac{2a}{2a - 1}  &\text{ } \left( a \in \left(\frac{1}{2},1 \right)\right).
\end{dcases}
\end{align}
Theorem \ref{thm:2} with \eqref{diminishing} and \eqref{diminishing_0} thus indicates that algorithm \eqref{modified_H} using a diminishing step size defined by \eqref{modifined_a} and an increasing batch size defined by \eqref{increasing} has the following convergence rate:
\begin{align}\label{rate}
\min_{k \in [0:K-1]} \mathbb{E}[ f_0 (\bm{x}_k)]
= 
f_0^\star
+
\begin{dcases}
O \left( \frac{1}{K^a} \right) &\text{ } \left( a \in \left(0,\frac{1}{2} \right) \right)\\
O \left( \frac{\log K}{\sqrt{K}}  \right) &\text{ } \left( a = \frac{1}{2} \right)\\
O \left( \frac{1}{K^{1-a}} \right) &\text{ } \left( a \in \left(\frac{1}{2}, 1 \right) \right)\\
O \left( \frac{1}{\log K} \right) &\text{ } \left( a = 1 \right). 
\end{dcases}
\end{align} 

When we use a diminishing step size $\alpha_k$, we cannot use a constant batch size $b_k = b$, since $\alpha_k$ satisfying $1/b_k = 1/b \leq \alpha_k$ does not converge to $0$. Hence, Theorem \ref{thm:2} says that the use of constant batch sizes does not guarantee convergence of $\min_{k \in [0:K-1]} f_0 (\bm{x}_k)$ to $f_0^\star$, as is evident as well from Theorem \ref{thm:1}.

\section{Numerical Examples}\label{sec:4}
We considered the ERM problem to train Residual Network (ResNet-18) \cite{he2016} on the CIFAR100 dataset \cite{krizhevsky2009learning} ($n = 50,000$ and $d \approx 11.2$M; see also Section \ref{sec:1.3.1} for ERM problems in deep neural networks).
The experimental environment was two NVIDIA GeForce RTX 4090 GPUs and Intel Core i9 13900KF CPU. The software environment was Python 3.10.12, PyTorch 2.1.0, and CUDA 12.2. 

The standard algorithm for the ERM problem is SGD defined by \eqref{SGD_0}, i.e., 
\begin{align}\label{SGD_1}
\begin{split}
\bm{x}_{k+1}
&= T_{\bm{\xi}_k} (\bm{x}_k)
= \frac{1}{b_k} \sum_{i=1}^{b_k} T_{\xi_{k,i}} (\bm{x}_k) 
\coloneqq
\frac{1}{b_k} \sum_{i=1}^{b_k} 
\underbrace{\left( \mathrm{Id} - \eta \nabla f_{\xi_{k,i}} \right)}_{T_{\xi_{k,i}}} (\bm{x}_k)\\
&= \bm{x}_k - \frac{\eta}{b_k} \sum_{i=1}^{b_k} \nabla f_{\xi_{k,i}} (\bm{x}_k) 
= \bm{x}_k - \eta \nabla f_{\bm{\xi}_k}(\bm{x}_k),
\end{split}
\end{align}
where $\eta = 0.1$, $b = 2^7$, $b_0 = 2^3$, $\delta = 2$, $E = 30$, and  
\begin{align*}
b_k = 
\begin{dcases}
\underbrace{b_0, \cdots, b_0}_{E \text{ epochs}}, 
\underbrace{b_0 \delta, \cdots, b_0 \delta}_{E \text{ epochs}},
\underbrace{b_0 \delta^2, \cdots, b_0 \delta^2}_{E \text{ epochs}},
\cdots
&\text{ (Exponential increasing batch size)}\\
b &\text{ (Constant batch size)} 
\end{dcases}
\end{align*}
were used in the experiment.
Although the use of constant batch sizes does not always guarantee convergence of SGD \eqref{SGD_1} in theory (see, e.g., \cite[Section 3.1]{umeda2025increasing} and also the discussion in \eqref{increasing}), 
constant batch sizes are well used in practice.
The exponential increasing batch size \eqref{increasing} guarantees convergence of SGD \eqref{SGD_1} in theory (see, e.g., \cite[Section 3.2]{umeda2025increasing} and also the discussion in \eqref{increasing}).

We used the following algorithms to train ResNet-18 on the CIFAR100 dataset:\\
\textbf{[Krasnosel'ski\u\i-Mann]} 
The mini-batch stochastic Krasnosel'ski\u\i-Mann algorithm \eqref{K_M_1}
with $T_{\bm{\xi}_k} (\bm{x}_k)$ in \eqref{SGD_1} and $\alpha_k = 1/2$:
\begin{align}\label{KM_SGD}
\bm{x}_{k+1} 
= \frac{1}{2} \bm{x}_k + \frac{1}{2} T_{\bm{\xi}_k} (\bm{x}_k)
= \frac{1}{2} \bm{x}_k + \frac{1}{2} \left( \bm{x}_k - \eta \nabla f_{\bm{\xi}_k} (\bm{x}_k) \right)
= \bm{x}_k - \frac{\eta}{2} \nabla f_{\bm{\xi}_k} (\bm{x}_k).
\end{align}
\textbf{[Halpern]} 
The mini-batch stochastic Halpern algorithm (Algorithm \ref{algo:1})
with $T_{\bm{\xi}_k} (\bm{x}_k)$ in \eqref{SGD_1}, $\bm{x}_0 = \bm{0}$,
$c = 10^{-3}$, and $\alpha_k = c/(k +1)^a$ $(a = 1/4, 1/2, 3/4, 1)$:
\begin{align}\label{H_SGD}
\begin{split}
\bm{x}_{k+1} 
&= \frac{c}{(k +1)^a} \bm{x}_0 + \left( 1 - \frac{c}{(k +1)^a}  \right) T_{\bm{\xi}_k} (\bm{x}_k)\\
&= \left( 1 - \frac{c}{(k +1)^a}  \right) \left( \bm{x}_k - \eta \nabla f_{\bm{\xi}_k} (\bm{x}_k) \right)\\
&= \left( 1 - \frac{c}{(k +1)^a}  \right) \bm{x}_k -  \eta \left( 1 - \frac{c}{(k +1)^a}  \right) \nabla f_{\bm{\xi}_k} (\bm{x}_k).
\end{split}
\end{align}
We can check that Algorithm \eqref{KM_SGD} is SGD with a step size $\eta/2$, while the update rule of Algorithm \eqref{H_SGD} closely resembles that of SGD with weight decay \cite[Algorithm 1]{loshchilov2018decoupled} defined by 
$\bm{x}_{k+1} = (1 - c \eta_k) \bm{x}_k - \eta_k \nabla f_{\bm{\xi}_k} (\bm{x}_k)$. 
Since weight decay is commonly employed to improve generalization performance (i.e., the test accuracies), incorporating weight decay into Algorithm \eqref{H_SGD} is expected to improve its generalization performance over that of Algorithm \eqref{KM_SGD}.

\begin{minipage}[t]{0.49\linewidth}
\begin{figure}[H]
 \centering
 \begin{minipage}{1.0\linewidth}
 \centering
 \includegraphics[width=\linewidth]{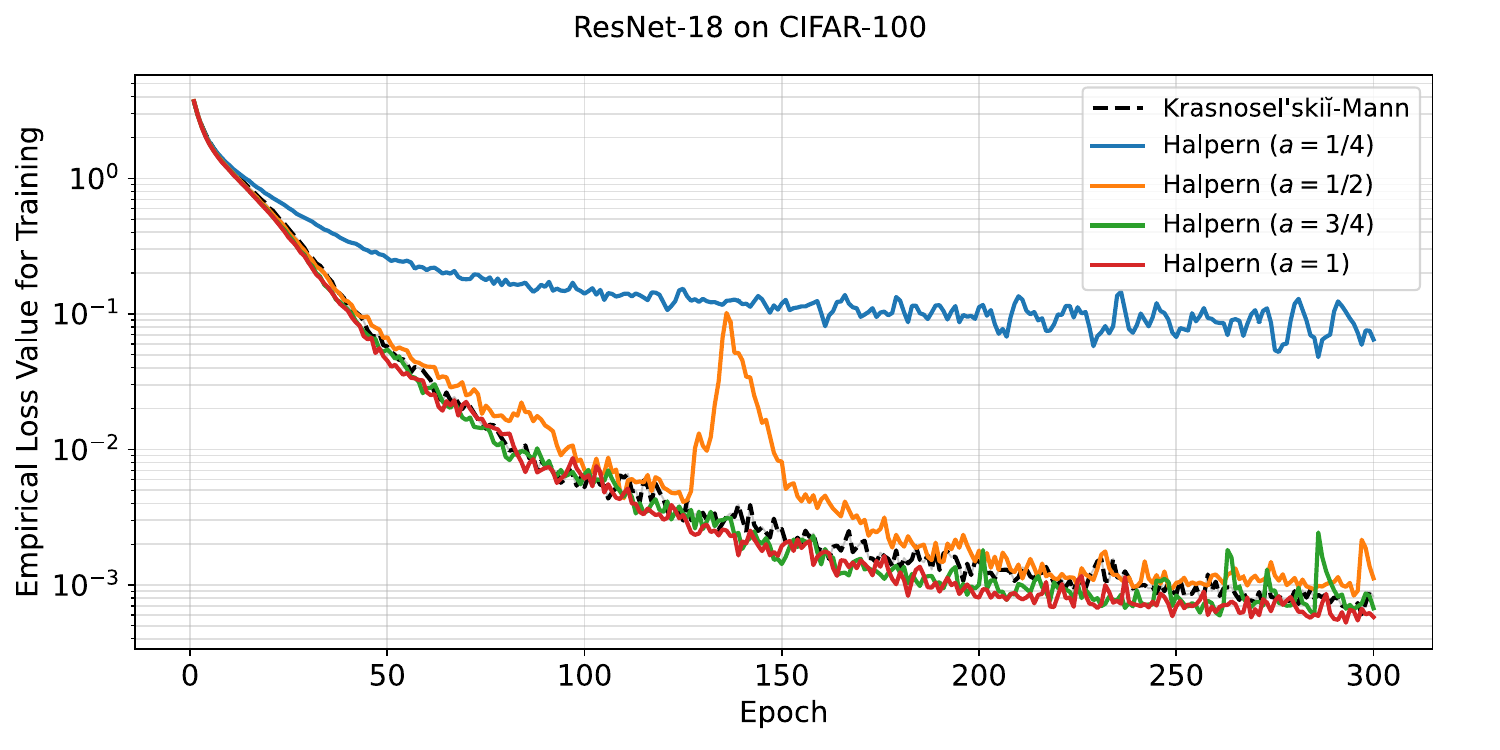}
 \end{minipage}
 \begin{minipage}{1.0\linewidth}
 \centering
 \includegraphics[width=\linewidth]{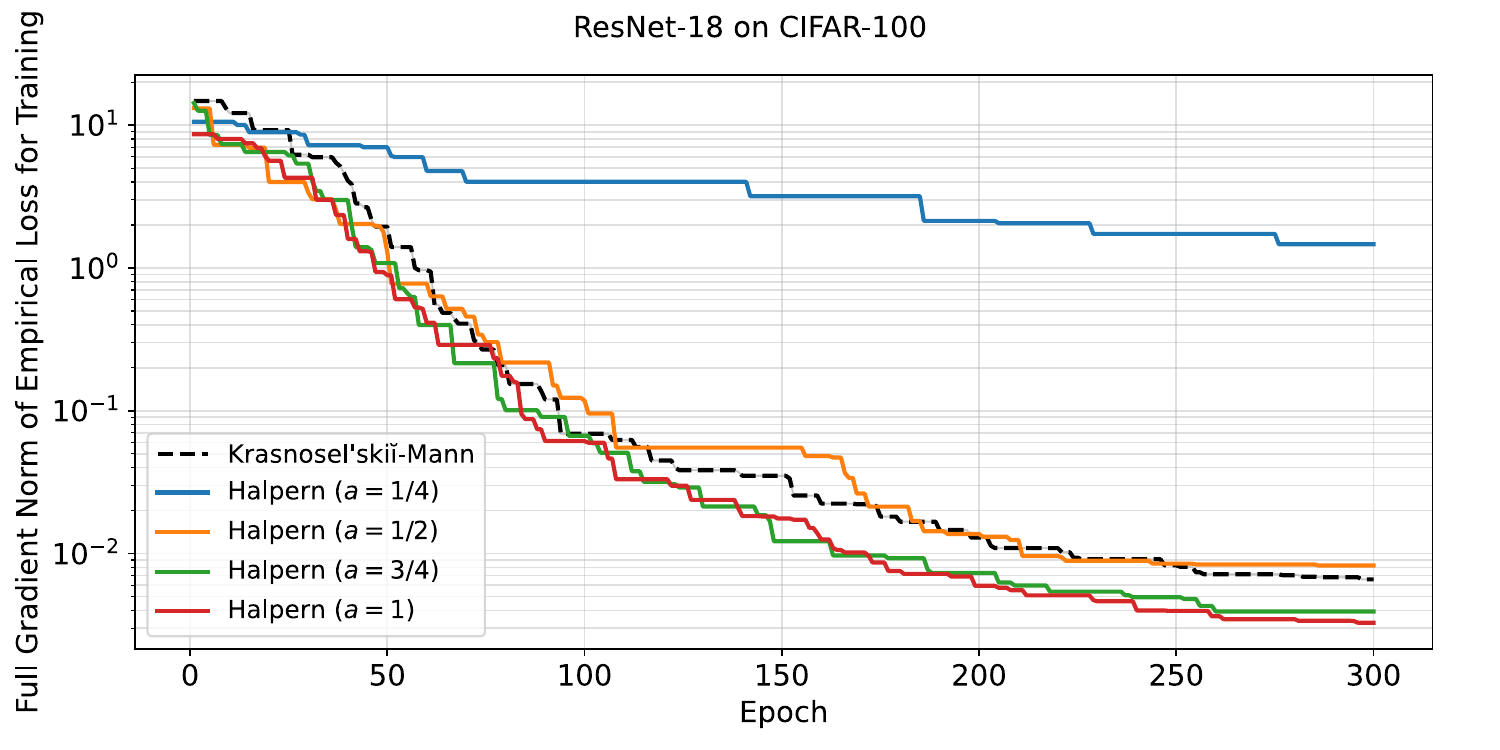}
 \end{minipage}
 \begin{minipage}{1.0\linewidth}
 \centering
 \includegraphics[width=\linewidth]{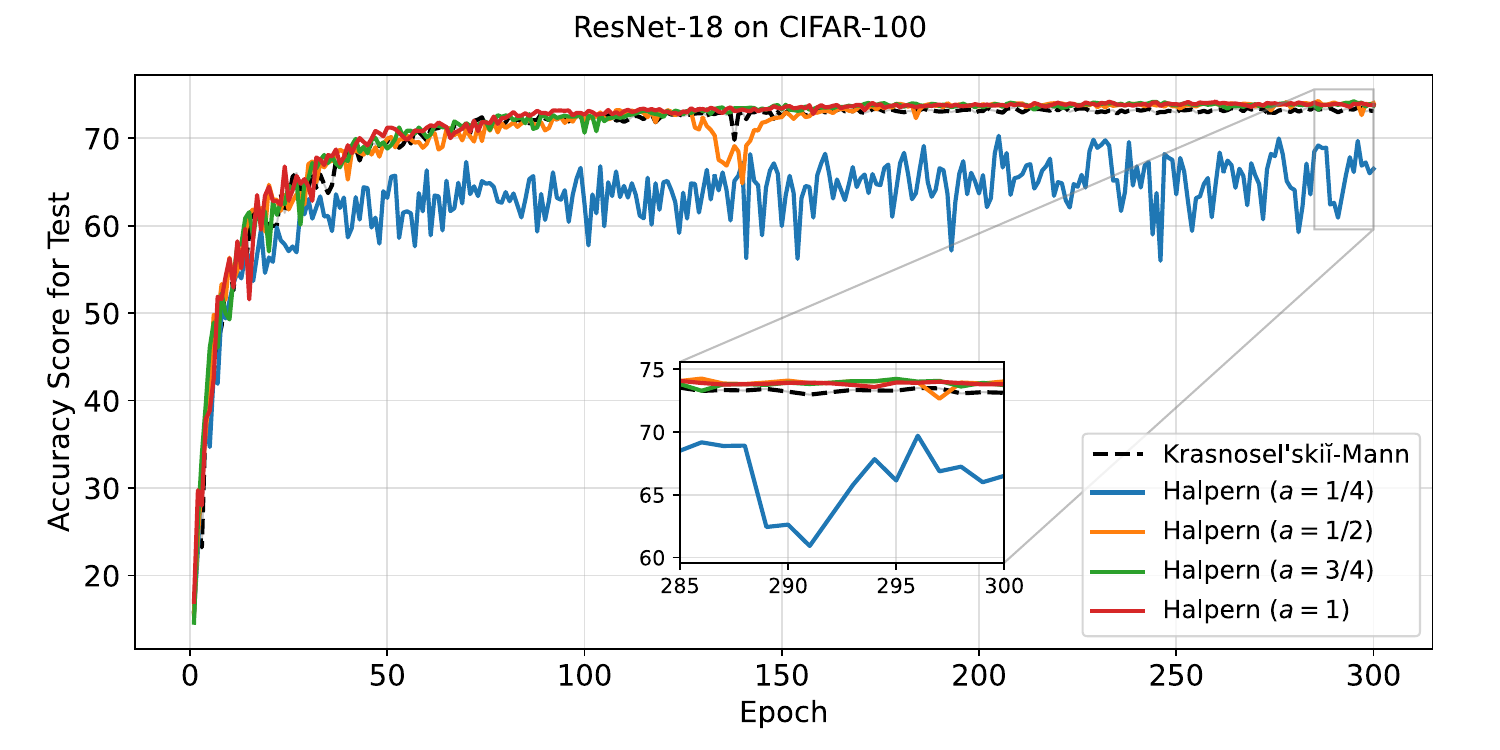}
 \end{minipage}
 \caption{
Empirical loss value, full gradient norm of empirical loss, and  accuracy score in testing for Algorithm \eqref{KM_SGD} ([Krasnosel'ski\u\i-Mann]) and Algorithm \eqref{H_SGD} ([Halpern] $(a = 1/4, 1/2, 3/4, 1)$) to train ResNet-18 on CIFAR100 dataset when batch size is constant $(b_k = b)$}
 \label{fig:constant_batch}
\end{figure}
\end{minipage}
\hfill
\begin{minipage}[t]{0.49\linewidth}
\begin{figure}[H]
 \centering
 \begin{minipage}{1.0\linewidth}
 \centering
 \includegraphics[width=\linewidth]{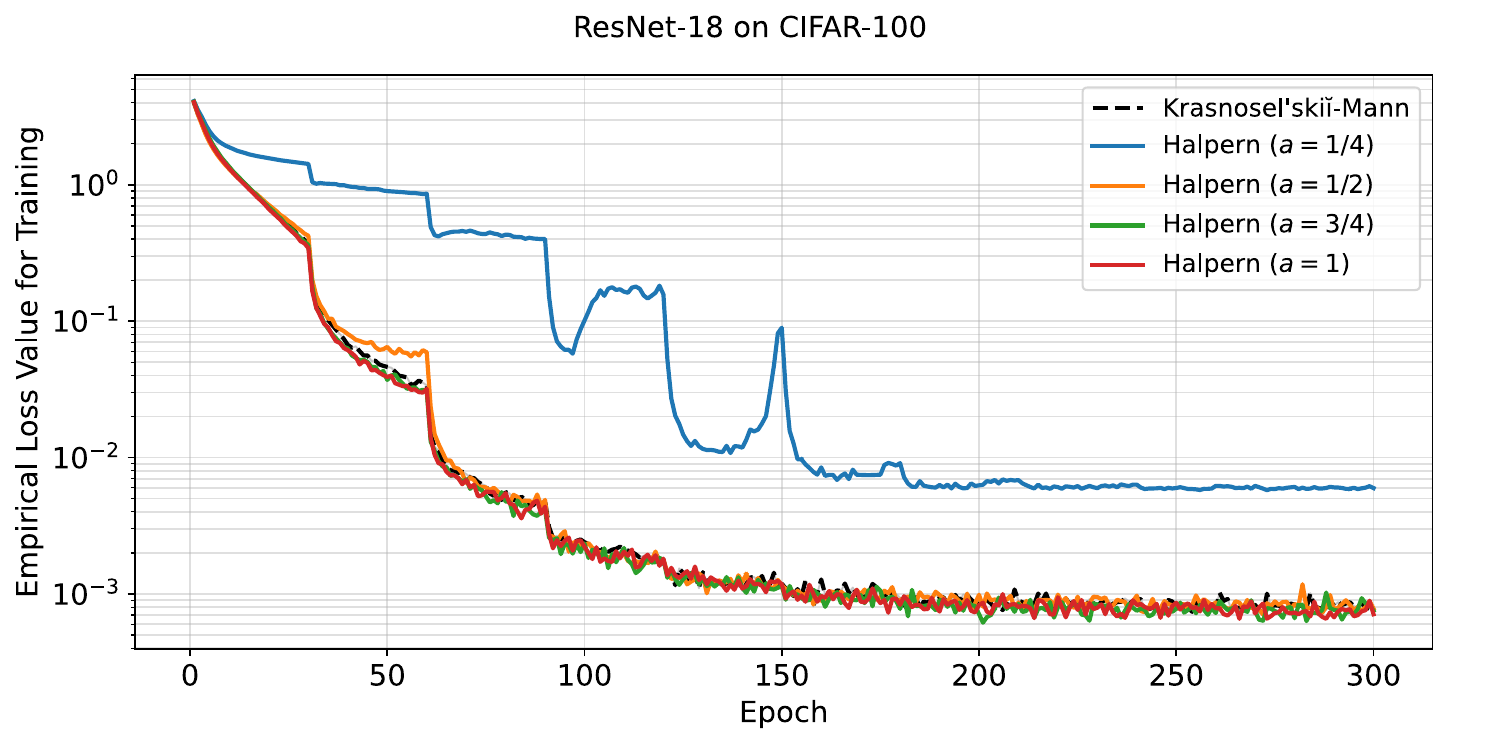}
 \end{minipage}
 \begin{minipage}{1.0\linewidth}
 \centering
 \includegraphics[width=\linewidth]{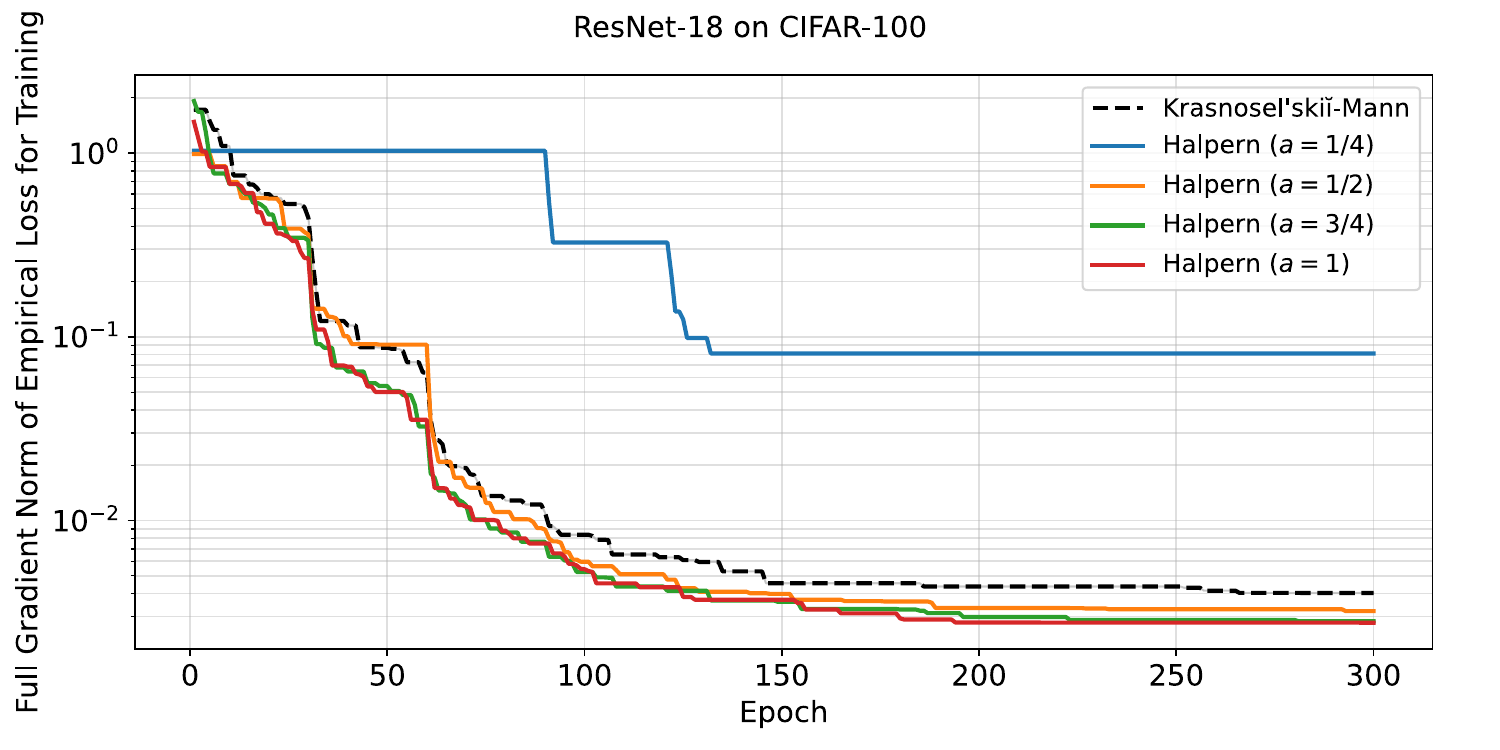}
 \end{minipage}
 \begin{minipage}{1.0\linewidth}
 \centering
 \includegraphics[width=\linewidth]{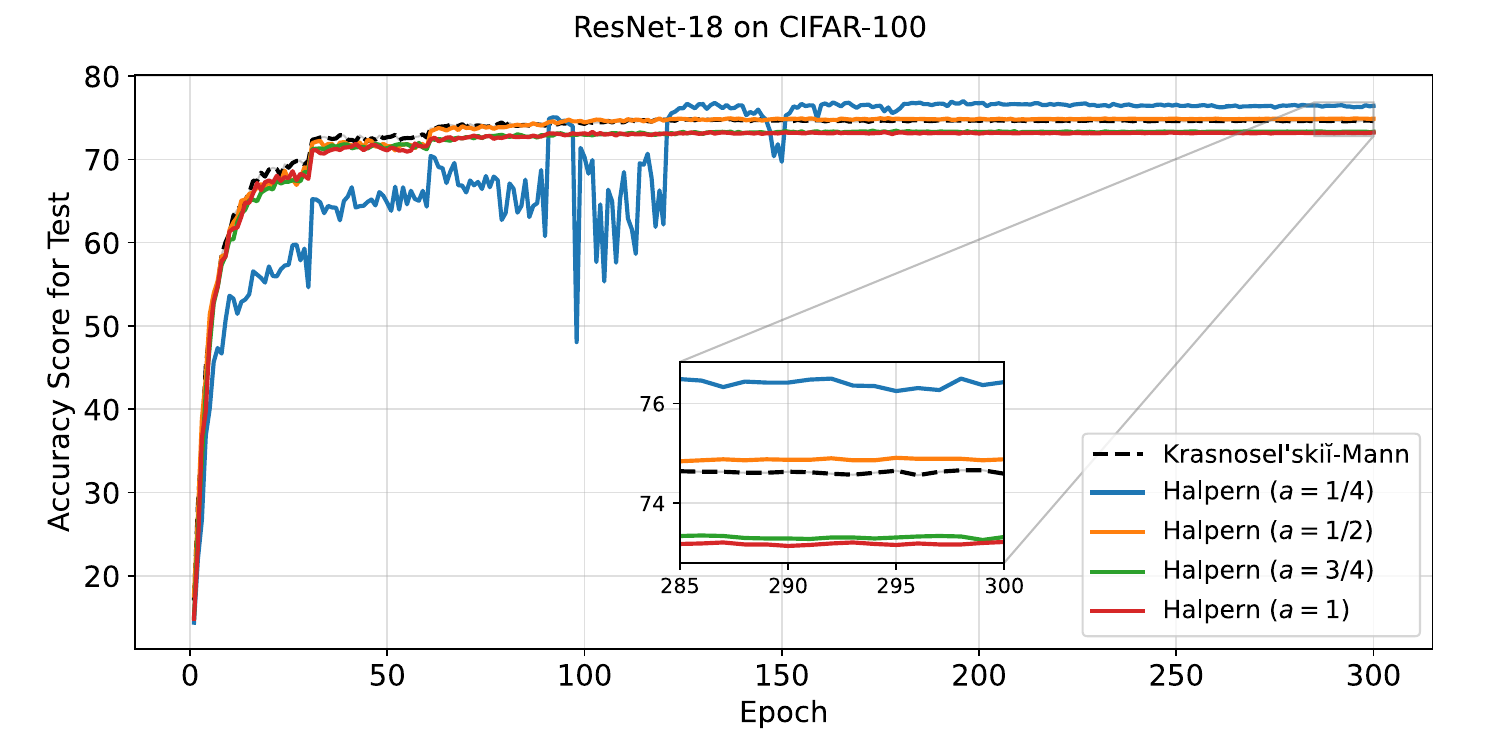}
 \end{minipage}
 \caption{Empirical loss value, full gradient norm of empirical loss, and  accuracy score in testing for Algorithm \eqref{KM_SGD} ([Krasnosel'ski\u\i-Mann]) and Algorithm \eqref{H_SGD} ([Halpern] $(a = 1/4, 1/2, 3/4, 1)$) to train ResNet-18 on CIFAR100 dataset when batch size is increasing $(b_k = b_0 \delta^k)$}
 \label{fig:increasing_batch}
\end{figure}
\end{minipage}

We evaluated three performance measures, 
the vale of the empirical loss $f(\bm{x}_k)$, 
the minimum of the full gradient norms $\min_{k \in [0:K-1]} \|\nabla f(\bm{x}_k)\|$, and the test accuracy.
If $f(\bm{x}_k)$ and $\|\nabla f(\bm{x}_k)\|$ decrease when $k$ increases, then $(\bm{x}_k)$ converges to a minimizer of $f$.
That is, the algorithm generating the sequence $(\bm{x}_k)$ can train ResNet-18 on the CIFAR100 dataset accordingly.
In machine learning, the test accuracy defined by 
((the number of correctly classified test samples)/(the total number of test samples)) $\times 100 
=$
((the number of correctly classified test samples)/10,000) $\times 100$
is a key metric for evaluating the generalization performance of the deep neural network model.

Figure \ref{fig:constant_batch} shows the behavior of Algorithm \eqref{KM_SGD} ([Krasnosel'ski\u\i--Mann]) and Algorithm \eqref{H_SGD} ([Halpern] with $a=1/4,1/2,3/4,1$) when the batch size is fixed at $b_k = b = 2^7$.
Although a constant batch size $b_k = b = 2^7$ does not guarantee the convergence of either algorithm in theory, Figure \ref{fig:constant_batch} shows that all algorithms except [Halpern] ($a=1/4$) steadily decreased both $f(\bm{x}_k)$ and $\|\nabla f (\bm{x}_k)\|$, while achieving high test accuracies.
In particular, [Halpern] ($a=1/2,3/4,1$) consistently achieved higher test accuracies than [Krasnosel'ski\u\--Mann].

Figure \ref{fig:increasing_batch} shows the behavior of Algorithm \eqref{KM_SGD} ([Krasnosel'ski\u\i--Mann]) and Algorithm \eqref{H_SGD} ([Halpern] with $a=1/4,1/2,3/4,1$) when the batch size increases exponentially, that is, $b_k$ is doubled every $E$ epochs.
A comparison of Figures \ref{fig:constant_batch} and \ref{fig:increasing_batch} shows that using the exponentially increasing batch size enabled the algorithms to decrease both $f(\bm{x}_k)$ and $\|\nabla f (\bm{x}_k)\|$ more rapidly than using the constant batch size $b_k = b$.
These results suggest that employing an increasing batch-size strategy is desirable not only from a theoretical standpoint, where it guarantees convergence, but also from a practical perspective because it accelerates optimization.
Moreover, Figure \ref{fig:increasing_batch} shows that [Halpern] ($a=1/4, 1/2$) achieved higher test accuracies than [Krasnosel'ski\u\i--Mann].

\section{Conclusion and Future Work}
In the present study, a stochastic fixed point problem for nonexpansive mappings was considered and a convergence analysis of the mini-batch stochastic Halpern algorithm for solving this problem was performed. It was shown that, for a decreasing step size and increasing batch size, the algorithm could achieve mean-square convergence to the point in the fixed point set closest to some given point. The convergence rates under certain additional conditions could also be described.
Furthermore, numerical experiments demonstrated the algorithm's superior empirical performance.

The present study assumed that the variance of a stochastic mapping was bounded (Assumption (A2)), but previous reports \cite{pmlr-v97-simsekli19a,pmlr-v139-garg21b,pmlr-v238-battash24a,ahn2024linear} have shown that stochastic noise in practical machine learning may include heavy-tailed behavior, which violates this assumption. Therefore, it is also necessary to verify whether the mini-batch stochastic Halpern algorithm converges in the presence of heavy-tailed behavior.

\section*{Declarations}
\subsection*{Ethics approval and consent to participate}
This study does not involve human participants or animals; therefore, ethics approval and consent to participate were not required.

\subsection*{Consent for publication}
This study does not contain any individual person's data; therefore, consent for publication is not applicable.

\subsection*{Availability of data and materials}
No datasets were generated or analyzed during the current study.

\subsection*{Competing interests}
The author declares that there are no competing interests.

\subsection*{Funding}
This work was supported by a JSPS KAKENHI Grant, Number 24K14846. 

\subsection*{Author's contributions}
The author conceived the study, developed the methodology, performed the analysis, and wrote the manuscript.

\subsection*{Acknowledgments}
I am sincerely grateful to the Editor, Reza Saadati, and the two anonymous reviewers for helping me improve the original manuscript. 
I also thank Hikaru Umeda for his input on the numerical examples.




\end{document}